\documentclass{article}
\usepackage[T1]{fontenc}
\usepackage{amsmath}
\usepackage{amssymb}
\usepackage{amsthm}
\usepackage{color}
\usepackage{graphicx}
\usepackage{algpseudocode} 
\usepackage{algorithm} 
\usepackage[]{subfig}
\usepackage[retainorgcmds]{IEEEtrantools}
\usepackage{accents}
\newcommand{\U}[1]{U_{#1}}

\renewcommand{\S}{\mathcal{S}}
\newcommand{\F}{ \boldsymbol{G} }
\renewcommand{\vec}{\mathrm{vec}}
\newcommand{\z}{\phantom{0}}
\renewcommand{\d}{\mathrm{d}}
\newcommand{\vect}[1]{\boldsymbol{#1}}
\newcommand{\vel}[1]{\mathbf{#1}}
\newtheorem{thm}{Theorem}

\newtheorem{remark}{Remark}

\newcommand{\buniv}{\hat{b}}

\newcommand{\Bmull}[2]{\hat{B}_{#2}^{#1}}
\newcommand{\bunivv}[2]{\hat{b}_{#2}^{#1}}

\newcommand{\dimp}[1]{m_{\alpha_{#1}}^{p}}
\newcommand{\dimpu}[1]{m_{\alpha_{#1}+1}^{p+1}}

\newcommand{\maten}[3]{\left[#1\right]_{#2,#3}}

\newcommand{\geo}[1]{{\mathcal{P}}_{#1}^{\F}}

\numberwithin{equation}{section}

\newcommand{\precVg}{{P}_V^{\F}}

\newcommand{\precQg}{{P}_Q^{\F}}

\newcommand{\precVgg}{\accentset{\frown}{P}_V}
\newcommand{\precVggc}[1]{\accentset{\frown}{{P}}_{V,#1}}

\def\GS{\color{black}} 
\def\B{\color{black}}
\def\MT{\color{black}}
\def\MT2{\color{red}}

\begin{document}

\pagestyle{myheadings}
\markboth{M. Montardini, G. Sangalli and M. Tani}{}

\title {Robust isogeometric preconditioners for the Stokes system based on the Fast Diagonalization method \thanks{Version of April 10, 2018}}

\author{M. Montardini \thanks{Universit\`a di Pavia, Dipartimento di Matematica ``F. Casorati'', 
Via A. Ferrata 1, 27100 Pavia, Italy. } \and
  G. Sangalli$^{\dag}$\thanks{IMATI-CNR ``Enrico Magenes'',  Pavia, Italy. \vskip 1mm \noindent Emails: 
{\tt  monica.montardini01@universitadipavia.it,  \{giancarlo.sangalli, mattia.tani\}@unipv.it}}
\and  M. Tani $^{\dag}$}

%\date{\documentdate}
%
\maketitle

\begin{abstract}

In this paper we propose  a new class of preconditioners for
the isogeometric discretization of the Stokes system. 
%Their application involves the solution of a Sylvester-like equation or a linear system with a Kronecker matrix. 
Their application involves the solution of a Sylvester-like equation,
which can be done efficiently thanks to the Fast Diagonalization
method. \GS These preconditioners  are robust with respect to both the
spline degree and mesh size.  By  incorporating  information on the
geometry parametrization  and  equation coefficients, we maintain efficiency on non-trivial
computational domains  and for variable kinematic
viscosity. \B  In our numerical tests we compare to a standard approach, showing that  the overall iterative solver based on our preconditioners is significantly faster.
% Furthermore, in  our approach the
% solver time, for high-degree high-regularity splines
%(the so-called $k$-method),  is dominated by the matrix-vector multiplication, indicating that further significant
%improvements  are possible by shifting to  a matrix-free implementation.

\vskip 1mm
\noindent
{\bf Keywords:} Isogeometric analysis, $k$-method, preconditioning, Stokes system, tensor product,
Kronecker product.
 \end{abstract}

\section{Introduction}

Isogeometric analysis (IGA) has been introduced by T.J.R. Hughes et al. in
the seminal paper \cite{Hughes2005}. IGA is   an innovative numerical method
to discretize  partial differential equations (PDEs), 
 based on using the same functions that
describe the computational domain in computer-aided
design (CAD) systems also for the representation of the solution. These functions are  B-Splines or NURBS or
generalizations of them\B.
For a complete description of the method and an overview of various
engineering applications, see \cite{Cottrell2009}. For a
mathematical-oriented  overview of IGA we refer to \cite{acta-IGA}.

 IGA is a   high-order numerical method, when
 high-degree polynomial/spline approximation is adopted. However  within
 IGA there is the 
possibility of high-regularity approximating
 functions. The typical case is indeed when splines of degree $p$ and global $C^{p-1}$
regularity are used within each patch. This is called   the
isogeometric $k$-method, which presents significant  advantages in comparison to $C^0$ finite
elements of degree $p$,  from many points of view:   higher accuracy per degree-of-freedom 
(see \cite{Evans_Bazilevs_Babuska_Hughes,acta-IGA}), improved spectral
behaviour   (see \cite{HRS08}), the possibility of dealing directly with
higher-order PDEs (\cite{GCBH08} is the first paper in this direction) or constructing
smooth structure-preserving schemes (see \cite{Buffa2011}).

  In this paper  the problem of interest is the  Stokes
system. We consider in particular two well-known  isogeometric
discretizations for which stability and convergence is known. 
One is the  extension of the Taylor-Hood 
element, which is \emph{inf-sup} stable, see 
\cite{Bazilevs2006,Buffa2011,Buffa2011b,Bressan2010,Bressan2012}. The other is the  extension
of the Raviart-Thomas element, which is stable and
structure-preserving, in the sense that the discrete solution is
pointwise divergence-free; see \cite{Buffa2011b,Evans2013} (and
\cite{Evans2013b,Evans2013c} for its extension to
Navier-Stokes). Both allow for arbitrary degree and regularity, in the spirit of
the $k$-method. 

The $k$-method is not costless: the computational
cost  per degree-of-freedom 
when  dealing with the $k$-method linear system grows as the degree and
regularity increase. In this paper we focus on  the cost of solving the
system, which is only one part of the problem (the other important part is the formation of the system
matrix, which is also an active research field).  Linear solvers that are developed for finite elements (e.g., direct
\cite{Collier2012}, iterative  multilevel \cite{Buffa2013}) work well
for low-degree isogeometric analysis but the computational performance
deteriorates for the high-degree $k$-method. Recently,  papers have appeared with preconditioners that behave
robustly for the isogeometric $k$-method: \cite{BeiraodaVeiga2013}
adopts a domain-decomposition approach,  \cite{Donatelli2015} and
\cite{Hofreither2017} are based on the multigrid idea (in particular,
the latter contains a proof of robustness, based on the
theory of 
{ \cite{Hofreither2017}}), and finally \cite{Sangalli2016}, which
uses a direct solver at the preconditioner stage, and takes
advantage of the tensor-product structure of the multivariate spline
spaces. All these  papers deal with the Poisson problem. 

Isogeometric preconditioners for the
Stokes system have also been studied in recent papers:
\cite{Cortes2015,Cortes2017}  consider block-diagonal and 
block-triangular preconditioners combined to  black-box solvers (either
algebraic-multigrid or incomplete factorization);  \cite{Pavarino2016}
studies the domain-decomposition FETI-DP  strategy;  \cite{Coley2017}
{ focuses} on  a multigrid strategy;
{ another multigrid approach, which extends the results of { \cite{Hofreither2017}}, can be found in \cite{Takacs2017}.}
%In this work we extend \cite{Sangalli2016} to the (stationary) Stokes
%system.

In the present work, for both Taylor-Hood  and   Raviart-Thomas isogeometric
discretizations of the Stokes system, we consider  preconditioners having the classical
{ block structure} (see \cite{Elman2014}) and using direct solvers to
invert the  diagonal blocks. 

In the simplest approach, our pressure Schur complement preconditioner is
the  pressure mass matrix in parametric coordinates, which is solved
by  exploiting   its
Kronecker structure. Moreover,  our
preconditioner  for  the velocity blocks is a component-wise
Laplacian in parametric coordinates, and  its solution is the
solution of a Sylvester-like equation. The latter equation is well studied
in the numerical linear algebra community (see for example the
overview \cite{Simoncini2016}); among many methods, following
\cite{Sangalli2016} we adopt a direct solver named Fast
Diagonalization (FD) method, see \cite{Deville2002,Lynch1964}. 

An important problem we have to face is  the treatment of the
geometry parametrization.  The simplest approach outlined above does not incorporate any geometry information in
the preconditioner, causing  a significant loss of
efficiency on complex geometry parametrizations. To overcome this
limitation,   we propose a modification of the preconditioner for a
partial inclusion of the geometry information, without increasing its
computational cost. Even though the mathematical analysis of this modification is postponed
to  a later work, in our numerical benchmarking  we show the clear benefits of this
approach. 
Indeed, we show theoretically and numerically that our preconditioner is robust with respect to the mesh size $h$ and spline degree $p$, both for the isogeometric  Taylor-Hood  and Raviart-Thomas methods. 
%While previous papers considered  low-degree splines only ($p=2,3$ typically), our numerical tests  on three-dimensional (3D) problems
%evidence a computational cost of the preconditioner  that is almost independent of $p$ (tested up to $p=5$, for memory constraints). 
While previous papers considered  low-degree splines only
  (typically quadratics and cubics), we are motivated to consider
  higher degrees in our tests (up to degree $6$ for the velocity and
  $5$ for the pressure,  for memory constraints) by the fact that the
  computational cost of our preconditioner is almost independent of
  the degree. The iterative solver total  computational time  is $O(n_{dof} p^3)$, but  it is
 heavily dominated by the
 matrix-vector multiplication which takes more than the
 $99\%$ of the overall cost when the pressure degree is $5$ and the
 velociy degree is $6$,  on a $16^3$ elements mesh. In this case our preconditioners is much faster than the
alternatives known in literature: for example, about $3$ orders of
magnitude when comparing to a standard preconditioner
based on the incomplete Cholesky factorization, which is
known to be  an  effective choice (see, e.g., \cite{Cortes2015}). 

In conclusion our  numerical benchmarks confirm that the proposed preconditioner
is very efficient and well suited for the $k$-method. 
%Further advances in the solver performance require at first a matrix-free approach, in order to accelerate  the matrix-vector  multiplication operation. 
Further advances in the solver performance can be achieved with a
matrix-free approach, that accelerates the  matrix-vector
multiplication operation,  for moderate or large degree. A first step in this research direction is \cite{Sangalli2017}.

The outline of the paper is as follows.
In Section 2 we give a short review of the  Taylor-Hood and Raviart-Thomas
isogeometric discretizations for the Stokes system, and  summarize  the main properties of the Kronecker product. 
The derivation of the discrete Stokes system is given in Section 3, while in
Section 4 we introduce some standard   block-structured  preconditioners  that we will consider in the numerical tests.
The core of the paper is Section 5, where we focus on the construction
of the preconditioning matrices for the velocity and pressure
blocks, discuss their properties and solution strategies. 
In the Section 6 we  propose the modification aimed at
improving the preconditioner efficiency by incorporating some information on the geometry parametrization.
Numerical results on three different single-patch domains are reported in Section 7.
Finally, in Section 8 we draw the conclusions and discuss future directions of research.

\section{Preliminaries} \label{sec:preliminaries}
\subsection{B-splines}
\label{sec:bas-isog-an}
In this section we summarize some basic concepts of B-spline based
isogeometric analysis, referring to \cite{Cottrell2009} for the details. 

Given $m$ and $p$ two positive integers, we introduce a \emph{knot vector} $\Xi:=\{0=\xi_1\leq ...\leq \xi_{m+p+1}=1\}$ and  the associated \emph{breakpoint vector}  $\mathcal{Z}:= \{\zeta_1,...,\zeta_s\}$, which
contains knots without repetitions.
We use \emph{open} knot vectors, i.e. we suppose $\xi_1=...=\xi_{p+1}=0$ and $\xi_m=...=\xi_{m+p+1}=1$ . 

 Then, according to Cox-De Boor recursion formulas \cite{DeBoor2001}, we
define univariate B-splines as:
\\
\indent
for $p=0$:
\begin{IEEEeqnarray*}{r"l}
\buniv_{ \vect{\alpha}, i}^{0}(\eta)= \begin{cases}1 &  { \textrm{if }} \xi_{i}\leq \eta<\xi_{i+1},\\
0 & \textrm{otherwise;}
\end{cases}
\end{IEEEeqnarray*}
\indent
for $p \geq 1$:
\begin{IEEEeqnarray*}{r"l}
\buniv_{ \vect{\alpha}, i}^{p}(\eta)=\! \begin{cases}\dfrac{\eta-\xi_{i}}{\xi_{i+p}-\xi_{i}}\buniv_{ \vect{\alpha}, i}^{p-1}(\eta)  \B +\dfrac{\xi_{i+p+1}-\eta}{\xi_{i+p+1}-\xi_{i+1}}\buniv_{ \vect{\alpha}, i+1}^{p-1}(\eta)  \B  & { \textrm{if }} \xi_{i}\leq  \eta<\xi_{i+p+1}, \\[8pt]
0  & \textrm{otherwise,}
\end{cases}
\end{IEEEeqnarray*}
where we adopt the convention $0/0=0$ and  $\boldsymbol\alpha: =\{-1,\alpha_2,...,\alpha_{s-1},-1 \} $ represents the \emph{regularity vector}.
Therefore, B-splines are piecewise polynomials with $\alpha _i$ continuous derivatives at $\zeta_i$. 
 The sum of the continuity and the  multiplicity at a breakpoint  is equal to the degree $p$, see \cite{Cottrell2009}.

The corresponding \emph{ univariate spline space}  is defined as
\begin{equation*}
\hat{\S}^p_{\boldsymbol\alpha}: = \mathrm{span}\{\bunivv{p}{\vect{\alpha},i}\}_{i = 1}^m. 
\end{equation*}

  To simplify the notation, we assume throughout this paper that the knot vector is \emph{uniform}, i.e. with equally spaced breakpoints, and the mesh size is denoted by $h$.
  For the same reason, we consider uniform regularity $\boldsymbol\alpha  =\{-1,\alpha,...,\alpha,-1 \} $. 
  Then we use the notation $\hat{\S}^p_{\alpha}$, $\bunivv{p}{\alpha,i}$  and set $m_{\alpha}^p:=m=\mathrm{dim}(\hat{\S}^p_{\alpha})$.
The extension of this framework to  non-uniform knot vectors and arbitrary regularity
is trivial (see, in this context,  \cite[Remark 4.4]{Bressan2012} and
\cite{Buffa2011})  and is considered in our numerical tests. \B

We consider  multivariate B-splines as  tensor-products of 
univariate B-splines.  For 3D problems, the case we address in this paper, the
univariate knot vectors $\Xi_l:=\{\xi_{l,1},...,\xi_{l,m_{\alpha_l}^{p_l}+p_l+1}\}$ for $
l=1,2,3$ and degree indices  $\vect{p}=(p_1,p_2,p_3) $ are given and,   for a multi-index $\vect{i} = (i_1,i_2,i_3)$, the multivariate B-spline is defined as
\begin{equation*}
\Bmull{\vect{p}}{\vect{\alpha}, \vect{i}}(\vect{\eta}) : = \bunivv{p_1}{\alpha_1, i_1}(\eta_1)\bunivv{p_2}{\alpha_2, i_2}(\eta_2)\bunivv{p_3}{\alpha_2, i_3}(\eta_3)
\end{equation*}
 where $\vect{\eta} = (\eta_1, \eta_2, \eta_3)$, and the \emph{multivariate spline space}  as
\begin{equation*}
\hat{\S}_{ \alpha_1, \alpha_2, \alpha_3}^{p_1,p_2,p_3}:=\hat{\S}_ {\alpha_1}^{p_1}
\otimes\hat{\S}_{ \alpha_2}^{p_2}\otimes \hat{\S}_{ \alpha_3}^{p_3}=\mathrm{span}\{\Bmull{\vect{p}}{\vect{\alpha}, \vect{i}} \ | \ i_k = 1,..., m_{\alpha_k}^{p_k}; k=1,2,3 \}.
\end{equation*}

 Throughout this paper, we  refer to \emph{spline spaces} as spaces of splines  defined on the parametric domain $\widehat{\Omega}:=[0,1]^3$.

\subsection{Isogeometric spaces}
\label{sec:iga-spaces}
 Let the computational domain $\Omega \subset \mathbb{R}^3$
be given by a single-patch spline parametrization
$\F\in \hat{\S}_{ \alpha_1, \alpha_2, \alpha_3}^{p,p,p} \times \hat{\S}_{ \alpha_1, \alpha_2, \alpha_3}^{p,p,p} \times \hat{\S}_{ \alpha_1, \alpha_2, \alpha_3}^{p,p,p} $ %{\MT (secondo voi questa notazione per lo spazio di funzioni e' chiara?)} 
of degree $p$ in each parametric direction. 
%We assume $\F$ is nonsingular and in particular $\det(\F)>0$. 
We assume that $\F$ is nonsingular, in the sense that its Jacobian is
everywhere invertible.

\emph{Isogeometric spaces} over $\Omega$ are suitable push-forwards, through $\F$, of spline spaces. 
In particular, in the context of the Stokes system, we focus on  two  discretizations of isogeometric spaces that have been proposed in \cite{Bressan2012} and \cite{Buffa2011} respectively.  
 Their definition and properties are
summarized in this section, see
\cite{Bazilevs2006,Bressan2010,Bressan2012,Buffa2011,Evans2013} for further details.

\subsubsection{Taylor-Hood isogeometric spaces} \label{sec:tay_hoo} 
  The Taylor-Hood (TH) spline spaces are defined as
% \begin{IEEEeqnarray*}{rl}
% \hat{V}_h^{TH}&:= \hat{\S}^{p+1,p+1,p+1}_{\alpha_1, \alpha_2,\alpha_3} \times   \hat{\S}^{p+1,p+1,p+1}_{\alpha_1,\alpha_2,\alpha_3}  \times    \hat{\S}^{p+1,p+1,p+1}_{\alpha_1,\alpha_2,\alpha_3} \\
% \hat{Q}_h^{TH}&:= \hat{\S}^{p,p,p}_{\alpha_1,\alpha_2,\alpha_3}. %\label{eq:TH_space_par}
% \end{IEEEeqnarray*}
  \begin{displaymath}
    \begin{aligned}
      \hat{V}_h^{TH}&:= \hat{\S}^{p+1,p+1,p+1}_{\alpha_1, \alpha_2,\alpha_3} \times   \hat{\S}^{p+1,p+1,p+1}_{\alpha_1,\alpha_2,\alpha_3}  \times    \hat{\S}^{p+1,p+1,p+1}_{\alpha_1,\alpha_2,\alpha_3} \\
\hat{Q}_h^{TH}&:= \hat{\S}^{p,p,p}_{\alpha_1,\alpha_2,\alpha_3}. %\label{eq:TH_space_par}
    \end{aligned}
  \end{displaymath}
For the velocity space we will also need
$$
\hat{V}_{h,0}^{TH}:= \left\lbrace \hat{\vel{v}}_h\in \hat{V}_{h}^{TH}\ \middle|\ \hat{\vel{v}}_h = 0 \ \mathrm{on} \ \partial \widehat{\Omega} \right\rbrace.
$$
A basis for $\hat{V}_{h}^{TH}$ is 
\begin{equation*}
\left\lbrace \vel{e}_k \Bmull{\vect{p}+1}{\vect{\alpha}, \vect{i}} \ \middle|\  i_l=1,...,m_{\alpha_l}^{p+1}; \  k,l=1,2,3 \right\rbrace.
\end{equation*} 
where $\vect{p}+1:=(p+1, p+1, p+1)$ and $\vel{e}_k$ is the $k$-th canonical basis vector { of $\mathbb{R}^3$}.

A basis for $\hat{V}_{h,0}^{TH}$ is then
\begin{equation}
\left\lbrace \vel{e}_k\Bmull{\vect{p}+1}{\vect{\alpha}, \vect{i}} \ \middle|\  i_l=2,...,m_{\alpha_l}^{p+1}-1; \ k,l=1,2,3 \right\rbrace
\label{eq:V_TH_0_basis1}.
\end{equation}
To each multi-index $\vect{i}$ present in \eqref{eq:V_TH_0_basis1} we associate  a scalar index $i$, corresponding to the lexicographical ordering of the internal degrees of freedom, \B such that
$$i = i_1-1+(i_2-2)(m_{\alpha_1}^{p+1}-2)+(i_3-2)(m_{\alpha_1}^{p+1}-2)(m_{\alpha_2}^{p+1}-2)$$ 
and, with abuse of notation, we rewrite the basis of $\hat{V}_{h,0}^{TH}$ as
\begin{equation*}
\left\lbrace \vel{e}_k\Bmull{\vect{p}+1}{\vect{\alpha}, i} \ \middle|\  i=1,...,n_{V,k}^{TH}; \ k=1,2,3 \right\rbrace,
\end{equation*} 
where $ n_{V,1}^{TH}=n_{V,2}^{TH}=n_{V,3}^{TH}:=(m_{\alpha_1}^{p+1}-2)(m_{\alpha_2}^{p+1}-2)(m_{\alpha_3}^{p+1}-2).
$

A basis for $\hat{Q}_{h}^{TH}$ is
\begin{equation}
\left\lbrace \Bmull{\vect{p}}{\vect{\alpha}, \vect{i}} \ \middle|\  i_l=1,...,m_{\alpha_l}^{p};\ l=1,2,3 \right\rbrace.
\label{eq:Q_TH_basis1}
\end{equation}
To each multi-index $\vect{i}$ present in \eqref{eq:Q_TH_basis1} we associate  a scalar index $i$, corresponding to the lexicographical ordering of the internal degrees of freedom, such that
\begin{equation}
i = i_1+(i_2-1) m_{\alpha_1}^{p}+(i_3-1) m_{\alpha_1}^{p} m_{\alpha_2}^{p} 
\label{eq:index}
\end{equation}
and, with abuse of notation, we rewrite the basis of $\hat{Q}_{h}^{TH}$  as
\begin{equation}
\left\lbrace \Bmull{\vect{p}}{\vect{\alpha}, i} \ \middle|\  i=1,...,n_Q^{TH} \right\rbrace,
\label{eq:Q_TH_basis}
\end{equation}
where  
\begin{equation}
n_Q^{TH}:=\mathrm{dim}(\hat{Q}_{h}^{TH})=m_{\alpha_1}^{p}m_{\alpha_2}^{p}m_{\alpha_3}^{p}.
\label{eq:np}
\end{equation}
The TH isogeometric spaces are the isoparametric push-forwards  (see \cite{Bressan2012,Buffa2011}):
\begin{IEEEeqnarray}{l}
\IEEEyesnumber\label{eq:TH_basis} \IEEEyessubnumber*
V_{h,0}^{TH}:=\mathrm{span}\left\lbrace \phi_i^{k,TH}  :=   \vel{e}_k\Bmull{\vect{p}+1}{\vect{\alpha}, i}\circ\F^{-1} \ \middle|\  i=1,...,n_{V,k}^{TH}; \ k=1,2,3  \right\rbrace
\label{eq:TH_V_basis}\\
Q_{h}^{TH}:=\mathrm{span}\left\lbrace \rho_i^{TH}:= \Bmull{\vect{p}}{\vect{\alpha}, i}\circ\F^{-1}\ \middle|\  i=1,...,n_Q^{TH}\right\rbrace.
\label{eq:TH_Q_basis}
\end{IEEEeqnarray}
For the discrete variational formulation of the Stokes system we will also need the space
\begin{equation}
Q_{h,0}^{TH}:=\left\lbrace q \in Q^{TH}_h \ \bigg\vert \ \int_{\Omega} q\ \d\Omega = 0 \right\rbrace.
\label{eq:TH_Q0}
\end{equation}
\B

\subsubsection{Raviart-Thomas isogeometric spaces}
The Raviart-Thomas (RT) spline spaces are defined as
\begin{IEEEeqnarray*}{rl}
\hat{V}_h^{RT}& := \hat{\S}^{p+1,p,p}_{\alpha_1+1, \alpha_2,\alpha_3} \times \hat{\S}^{p,p+1,p}_{\alpha_1,\alpha_2+1,\alpha_3} \times \hat{\S}^{p,p,p+1}_{\alpha_1,\alpha_2,\alpha_3+1} \\
\hat{Q}_h^{RT} & :=  \hat{\S}^{p,p,p}_{\alpha_1,\alpha_2,\alpha_3}.
\end{IEEEeqnarray*} 
For the velocity space we will also need
$$
\hat{V}_{h,0}^{RT}:=\left\lbrace \vel{\hat{v}}_h\in \hat{V}_{h}^{RT} \ \middle|\  \hat{\vel{v}}_h\cdot \vel{n} =0 \ \mathrm{on} \ \partial\widehat{\Omega} \right\rbrace.
$$
A basis for $\hat{V}_h^{RT}$ is
\begin{align*}
& \left\lbrace \vel{e}_k \Bmull{\vect{p}+\vel{e}_k}{\vect{\alpha} +\vel{e}_k, \vect{i}} \ \middle|\  i_k=1,..., \dimpu{k};\ i_l=1,...,\dimp{l};\ l\neq k;\ l,k=1,2,3\right\rbrace ,
\end{align*}
where $\vect{p}+\vel{e}_1 = (p+1,\ p,\ p)$, $\vect{p}+\vel{e}_2 = (p,\ p+1,\ p)$, $\vect{p}+\vel{e}_3 = (p,\ p,\ p+1)$ and $\vect{\alpha}+\vel{e}_1 = (\alpha_1+1,\ \alpha_2,\ \alpha_3)$, $\vect{\alpha}+\vel{e}_2 = (\alpha_1,\ \alpha_2+1,\ \alpha_3)$, $\vect{\alpha}+\vel{e}_3 = (\alpha_1,\ \alpha_2,\ \alpha_3+1)$.

A basis for $\hat{V}_{h,0}^{RT}$ is then
\begin{align}
& \left\lbrace \vel{e}_k\Bmull{\vect{p}+\vel{e}_k}{\vect{\alpha} +\vel{e}_k, \vect{i}}\ \middle|\  i_k=2,..., \dimpu{k}-1;\ i_l=1,...,\dimp{l};\ l\neq k; \ l,k=1,2,3\right\rbrace.
\label{eq:basisRT_V1}
\end{align}
To each multi-index $\vect{i}$ present in \eqref{eq:basisRT_V1} we associate  a scalar index $i$, corresponding to the lexicographical ordering of the internal degrees of freedom,   such that
\begin{IEEEeqnarray*}{c}
 \mathrm{for} \ k=1 \quad  i = i_1-1+(i_2-1)(\dimpu{1}-2)+(i_3-1)(\dimpu{1}-2)\dimp{2},\\
  \mathrm{for} \ k=2 \quad  i = i_1+(i_2-2)\dimp{1}+(i_3-1)\dimp{1}(\dimpu{2}-2),\\
  \mathrm{for} \ k=3 \quad  i = i_1+(i_2-1)\dimp{1}+(i_3-2)\dimp{1}\dimp{2}
  \end{IEEEeqnarray*}
and, with abuse of notation, we rewrite the basis of $\hat{V}^{RT}_{h,0}$ as
\begin{align}
& \left\lbrace \vel{e}_k\Bmull{\vect{p}+\vel{e}_k}{\vect{\alpha} +\vel{e}_k, i} \ \middle|\ i=1,...,n_{V,k}^{RT};  \ k=1,2,3\right\rbrace,
\label{eq:basisRT_V0}
\end{align}
where
$$n^{RT}_{V,1}=(\dimpu{1}-2)\dimp{2}\dimp{3}, \quad\quad n^{RT}_{V,2}=\dimp{1}(\dimpu{2}-2)\dimp{3}, \quad\quad n^{RT}_{V,3}=\dimp{1}\dimp{2}(\dimpu{3}-2).$$
As $\hat{Q}_h^{RT}=\hat{Q}_h^{TH}$, a basis for  $\hat{Q}_h^{RT}$ is  \eqref{eq:Q_TH_basis} and its dimension is denoted by $n_{Q}^{RT}=n_{Q}^{TH}=n_Q$ (cfr. \eqref{eq:np}).

The RT isogeometric spaces  are defined  by suitable push-forwards:
\begin{IEEEeqnarray}{l}
\IEEEyesnumber\label{eq:RT_basis} \IEEEyessubnumber*
V_{h,0}^{RT}:=\mathrm{span}\left\lbrace \phi_i^{k,RT}:= \left( {\left(\mathrm{det}(J_{\F})\right)}^{-1}J_{\F}  \vel{e}_k\Bmull{\vect{p}+\vel{e}_k}{\vect{\alpha} +\vel{e}_k, i } \right) \circ\F^{-1}  \ \middle|\  i=1,...,n_{V,k}^{RT}; \ k=1,2,3 \right\rbrace
 \label{eq:RT_V_basis}\\
Q_{h}^{RT}:=\mathrm{span}\left\lbrace \rho_i^{RT}:= \left( {\left(\mathrm{det}(J_{\F})\right)}^{-1}\Bmull{\vect{p}}{\vect{\alpha} , i} \right)\circ\F^{-1}\ \middle|\  i=1,...,n_Q^{RT} \right\rbrace
\label{eq:RT_Q_basis}.
\end{IEEEeqnarray}
The push-forward employed for $V_{h,0}^{RT}$ is the Piola transform and its use is important to assure inf-sup stability, see \cite{Buffa2011} and Section \ref{sec:Stokes}. 

We remark that although in the parametric domain $\hat{Q}_h^{RT}=\hat{Q}_h^{TH}$, in general ${Q}_h^{RT}\neq {Q}_h^{TH}$.\B

 For the discrete variational formulation of the Stokes system we will also need the space
\begin{equation}
Q_{h,0}^{RT}:=\left\lbrace q \in Q^{RT}_h \ \bigg\vert \ \int_{\Omega} q\ \d\Omega = 0 \right\rbrace.
\label{eq:RT_Q0}
\end{equation}
\B
  
\subsection{The Kronecker product}
\label{sec:kronecker}
We restrict to the case of square matrices and we consider $A\in \mathbb{R}^{n_1\times n_1}$, $B\in \mathbb{R}^{n_2 \times n_2}$ and $C\in \mathbb{R}^{n_3\times n_3}$. 
The entries of the matrix $A$ are denoted with $[A]_{i,j}$. 

The \emph{Kronecker product} between $A$ and $B$ is defined as 
$$ A \otimes B := \begin{bmatrix} \maten{A}{1}{1} B & \ldots & \maten{A}{1}{n_1} B \\ \vdots & \ddots & \vdots \\ \maten{A}{n_1}{1}B & \ldots & \maten{A}{n_1}{n_1} B \end{bmatrix} \; \in \mathbb{R}^{n_1 n_2 \times n_1 n_2}.$$
This operation is associative: $
 A \otimes B \otimes C =  (A \otimes B) \otimes C =  A \otimes (B \otimes C).
$ 

Given a tensor  $\mathbb{W}\in \mathbb{R}^{n_1\times n_2 \times n_3}$, the vec operator converts  $\mathbb{W}$ to a vector $\mathrm{vec}(\mathbb{W}) \in \mathbb{R}^{n_1 n_2 n_3}$ as
$$
\left[\mathrm{vec}(\mathbb{W})\right]_{i}:=\maten{\mathbb{W}}{i_1,i_2}{i_3}
$$
where $i=i_1+(i_2-1)n_1+(i_3-1)n_1n_2$, for $i_k=1,...,n_k$ and $k=1,2,3$.

Let $Y_m\in\mathbb{R}^{t\times n_m}$ for $m=1,2,3$ be three matrices. 
The \emph{m-mode} product $\times_m$ gives the following tensors
$$
[\mathbb{W} \times _1 Y_1]_{i_1,i_2,i_3}:=\sum_{k=1}^{n_1}\maten{Y_1}{i_1}{k}[\mathbb{W}]_{k, i_2, i_3 }
\quad
[\mathbb{W} \times _2 Y_2]_{i_1,i_2,i_3}:=\sum_{k=1}^{n_2}\maten{Y_2}{i_2}{k}[\mathbb{W}]_{i_1, k, i_3}
$$
\vskip -3mm
$$
[\mathbb{W} \times _3 Y_3]_{i_1,i_2,i_3}:=\sum_{k=1}^{n_3}\maten{Y_3}{i_3}{k}[\mathbb{W}]_{i_1, i_2, k}.
$$
See \cite{Kolda2009} for more details.
 
 Being primarily interested in 3D problems, we will deal with matrices of the form $A\otimes B \otimes C$.
 We will need the following properties:
 \begin{itemize}
 % \item The Kronecker product is a bilinear and associative operation.
 %
 \item It holds
  \begin{equation} \label{eq:krontranspose} (A\otimes B \otimes C)^T = A^T \otimes B^T \otimes C^T. \end{equation}
 In particular, if $A$, $B$ and $C$ are symmetric, then also $A\otimes B \otimes C$ is symmetric.
 \item If $A,B$ and $C$ are nonsingular, then \begin{equation}\label{eq:kroninv}(A\otimes B \otimes C)^{-1}= A^{-1}\otimes B^{-1}\otimes C^{-1}.\end{equation}

%\item If $A$ has $n_a$ eigenvalues $\lambda_i$, $B$ has $n_b$  eigenvalues $\mu_j$ and $C$ has $n_c$ eigenvalues $\nu_k$, then  $A\otimes B \otimes C$ has $n_an_bn_c$ eigenvalues of  the form $\lambda_{i}\mu_j\nu_k $. \\
\item { Let $\lambda_1, \ldots, \lambda_{n_1}$ denote the eigenvalues of $A$, $\mu_1, \ldots, \mu_{n_2}$, denote the eigenvalues of $B$ and $\nu_1, \ldots, \nu_{n_3}$ denote the eigenvalues of $C$. Then the eigenvalues of $A\otimes B \otimes C$ are $\lambda_i \mu_j \nu_k$, $i = 1,\ldots,n_1$, $j = 1,\ldots,n_2$, $k = 1,\ldots,n_3$.}
 In particular, if $A$, $B$ and $C$ are positive definite, then also $A\otimes B \otimes C$ is positive definite. \\
\item If $\mathbb{X}\in \mathbb{R}^{n_1\times n_2 \times n_3}$, then 
\begin{equation}
\label{eq:kronvec}
(A\otimes B \otimes C) \vec(\mathbb{X})=\vec(\mathbb{X}\times_1 A \times_2 B \times_3 C ).
\end{equation}
Thanks to this property the matrix $A\otimes B \otimes C$ does not need to be
formed to compute a matrix-vector product, resulting in a significant saving of memory and  
{ floating point operations (FLOPs)}.
\item It holds (from \eqref{eq:kroninv} and \eqref{eq:kronvec}):
\begin{equation}
(A\otimes B \otimes C)^{-1}\vec(\mathbb{X}) = \vec(\mathbb{X}\times_1 A^{-1} \times_2 B^{-1} \times_3 C^{-1}).
\label{eq:kron_vec_sol}
\end{equation}

 \end{itemize}

\section{Isogeometric analysis of the Stokes system}\label{sec:Stokes}

The   Stokes  system reads as
\begin{IEEEeqnarray*}{r"ll}
%\IEEEyesnumber\label{eq:Stokes} \IEEEyessubnumber*
-\nabla \cdot(2 \nu \,\nabla^s \mathbf{u}) + \nabla p = \mathbf{f}  \ \ \  & \textrm{in}   \  & \Omega\\
\nabla \cdot \mathbf{u} = 0    \ \ \ & \textrm{in} \  & \Omega
\end{IEEEeqnarray*}
where $\nabla^s = \frac{1}{2}\left(
  \nabla + \nabla ^T \right)$, $\mathbf{u}$ is the velocity, $p$ is
the scalar pressure and $\nu >0$ is the kinematic viscosity. We
assume  $\nu\in L^\infty (\Omega)$ and $\mathbf{f} \in  \mathbf{L}^2
(\Omega)$.  
 We consider no-slip boundary conditions, that is  we impose $\mathbf{u}  = 0 \ \ \textrm{on} \ \partial\Omega$.
The pressure is determined up to a constant. 

The standard
(mixed) variational formulation of the problem reads: find $\mathbf{u}
\in \mathbf{H}^1_0(\Omega):=\{\mathbf{v} \in \mathbf{H}^1(\Omega) \mid \mathbf{v}=0 \ \mathrm{on} \ \partial\Omega \}$ and $p \in L^2_0(\Omega):=\{q \in L^2(\Omega) \mid \int_{\Omega}q\ \d\Omega =0 \}$ \B such that
\begin{IEEEeqnarray}{l"rcr}
\IEEEyesnumber\label{eq:mix-Stokes} \IEEEyessubnumber*
a(\vel{u},\vel{v}) + b (\vel{u}, p) = (\vel{f},\vel{v})_{L^2}  & \forall\, \mathbf{v}& \, \in  &\mathbf{H}^1_0(\Omega)\\
b(\vel{u}, q) = 0 & \forall \, q & \, \in & L^2_0(\Omega),
\end{IEEEeqnarray}
where $(\cdot,\cdot)_{L^2}$ denotes the $L^2$ scalar product while the bilinear forms $a(\cdot,\cdot)$ and $b(\cdot,\cdot)$ are defined as% $\forall \vel{v}, \vel{w}\in \mathbf{H}^1_0(\Omega)$ and $\forall q\in L^2_0(\Omega) $ as
\begin{align}
a(\vel{w},\vel{v})& =\int_{\Omega}2\nu\,\nabla^s\mathbf{w}:\nabla ^s \mathbf{v}\; \d \Omega \label{eq:a_form}\\
   b(\vel{v}, q)& = - \int_{\Omega} q\, \nabla \cdot \mathbf{v}\;\d\Omega  .\nonumber
\end{align}
%The problem above is well posed, and in particular an \textit{inf-sup} condition holds (see \cite{Brezzi1991} \footnote{\GS aggiornare... \B}):
%\begin{equation}
%\inf_{q\in L^2_0(\Omega)  \\ q\neq 0} \sup_{\mathbf{v}\in \mathbf{H}^1_0(\Omega)} \frac{|b(\mathbf{v},q)| }{\parallel\hskip -0.7mm \mathbf{v} \hskip -0.7mm \parallel_{\mathbf{H}^1} \parallel \hskip -0.7mm  q \hskip -0.7mm  \parallel_{L^2}}\geq \beta 
%\end{equation}
%where $\beta>0$ is a constant.

The isogeometric Taylor-Hood (TH) discretization of Stokes system is a
standard Galerkin method for \eqref{eq:mix-Stokes} and reads: find $\vel{u}^{TH}_h\in V_{h,0}^{TH}$ and $p^{TH}_h\in Q_{h,0}^{TH}$ such that
\begin{IEEEeqnarray}{l'lcr}
\IEEEyesnumber\label{eq:TH_dis} \IEEEyessubnumber*
a(\vel{u}^{TH}_h, \vel{v}_h) + b(\vel{v}_h, p^{TH}_h) = (f, \vel{v}_h)_{L^2} & \forall\, \mathbf{v}_h& \, \in  & V_{h,0}^{TH},\\
b(\vel{u}^{TH}_h, q_h) = 0 & \forall\, q_h & \, \in  & Q_{h,0}^{TH},
\end{IEEEeqnarray}
where $V_{h,0}^{TH}$ and $Q_{h,0}^{TH}$ are defined as \eqref{eq:TH_V_basis} and \eqref{eq:TH_Q0}. 
A detailed analysis on the well posedness of this problem can be found
in \cite{Bazilevs2006, Bressan2010, Bressan2012}.

The isogeometric Raviart-Thomas (RT)  discretization we adopt is based on  a
Nitsche formulation for the weak imposition of the tangential
Dirichlet boundary condition to ensure stability (see \cite{Evans2013}). 

The method reads:  find
$\vel{u}^{ RT}_h \in V_{h,0}^{RT}$ and $p ^{  RT}_h\in Q_{h,0}^{RT}$  such that
\begin{IEEEeqnarray}{l'lcr}
\IEEEyesnumber\label{eq:RT_dis} \IEEEyessubnumber*
a(\vel{u}^{RT}_h, \vel{v}_h) + \sigma(\vel{u}^{RT}_h, \vel{v}_h)+ b(\vel{v}_h, p ^{ RT}_h) = (f, \vel{v}_h)_{L^2} & \forall\, \mathbf{v}_h& \, \in  & V_{h,0}^{RT},\\
b(\vel{u}^{RT}_h, q_h) = 0 & \forall\, q_h & \, \in  & Q_{h,0}^{RT},  
\end{IEEEeqnarray}
where  $ V_{h,0}^{RT}$ and $Q_{h,0}^{RT}$ are defined as \eqref{eq:RT_V_basis} and \eqref{eq:RT_Q0} and the bilinear form $\sigma(\cdot, \cdot)$ is defined as
\begin{equation}
\sigma(\vel{w}_h, \vel{v}_h):=\int_{\partial\Omega}
2\nu\left[\frac{ C_{pen}}{h}\vel{w}_h\! \cdot\! \vel{v}_h  - \left(\left(\nabla^s\vel{w}_h\right)\vel{n}\right)\cdot\vel{v}_h -\left(\left(\nabla^s\vel{v}_h\right)\vel{n}\right)\cdot\vel{w}_h\right] \;\d\Gamma, \label{eq:theta}
\end{equation}
with $C_{pen}>0$ a penalty parameter.
 %We choose $C_{pen} = 5(\alpha + 1)$, as it numerically leads to stable schemes (see \cite{Evans2013}). \B
 The well-posedness of RT  discretization for Stokes problem and the choice of $C_{pen}$ are analysed in \cite{Evans2013}.

In practice, we build the linear system by replacing $Q_{h,0}^{TH}$ and $Q_{h,0}^{RT}$ by $Q_{h}^{TH}$ and $Q_{h}^{RT}$, respectively. This means that we do not incorporate the zero-mean-value constraint into the pressure space, since this will be taken care of by the Krylov iterative solver later.

%In practice, we build the linear system using in  \eqref{eq:TH_dis} and \eqref{eq:RT_dis}   the pairs $(V_{h,0}^{TH}, Q_{h}^{TH})$ and $(V_{h,0}^{RT},Q_{h}^{RT})$, thus incorporating suitable boundary conditions only for the velocity.
%We do not incorporate in the pressure space any zero-mean-value constraint, since this will be taken care of by the Krylov iterative solver later.
Then, the discrete Stokes system  matrix is
\begin{equation}
\mathcal{A} = 	\begin{bmatrix}
       A & B^T          \\[0.3em]
	B & 0
     \end{bmatrix},
     \label{eq:Stokes_system}
\end{equation}
where 

$$
A = \begin{bmatrix}
A_{11} &  A_{12} & A_{13}\\
A_{21} &  A_{22} & A_{23}\\
A_{31} &  A_{32} & A_{33}
\end{bmatrix},
\quad B=[B_1 \ B_2 \ B_3] ,
$$
and for TH discretization, $i= 1,...,n_{V,r}^{TH}$, $j= 1,...,n_{V,s}^{TH}$, $r,s=1,2,3$ and $l=1,...,n_Q$
\begin{IEEEeqnarray*}{rl}
&[A_{rs}^{TH}]_{i,j}   := a(\vect{\phi}_{i}^{r,TH},\vect{\phi}_{j}^{s,TH}), \\
&[B_r^{TH}]_{l,j}  := b(\vect{\phi}_{j}^{r,TH}, \rho_l^{TH}),
\end{IEEEeqnarray*}
while for RT discretization, $i=1,...,n_{V,r}^{RT}$, $j=1,...,n_{V,s}^{RT}$, $r,s=1,2,3$ and  $l=1,...,n_Q$
\begin{IEEEeqnarray*}{rl}
&[A_{rs}^{RT}]_{i,j}  := a(\vect{\phi}_{i}^{r,RT},\vect{\phi}_{j}^{s,RT})+\sigma(\vect{\phi}_{i}^{r,RT},\ \vect{\phi}_{j}^{s,RT}), \\
&[B_r^{RT}]_{l,j} := b(\vect{\phi}_{j}^{r,RT}, \rho_l^{RT}),
\end{IEEEeqnarray*}
referring to Section \ref{sec:iga-spaces} for the notations of the basis.

In particular we have that for $k=1,2,3$ %and $i,j = 1,...,n^{TH}$
\begin{IEEEeqnarray}{rcl}
\left[A^{TH}_{kk}\right]_{i,j} & = &  \int_{\widehat{\Omega}}  \left( \nabla \Bmull{\vect{p}+1}{\vect{\alpha}, i}\right)^T\ \mathfrak{{C}}^{TH}_k\ \nabla \Bmull{\vect{p}+1}{\vect{\alpha}, j}\;\d\boldsymbol\eta ,
 \label{eq:Akk_TH}
\end{IEEEeqnarray}
where
\begin{equation}
\mathfrak{C}_k^{TH}(\boldsymbol \eta) = \nu(J_{\F}^{-1}J_{\F}^{-T}+D_kD_k^T)\ |\mathrm{det}(J_{\F})| \label{eq:Q_TH}
\end{equation}
and $D_k:=J_{\F}^{-1} \vel{e}_k$, while %for $k=1,2,3$ and $i,j = 1,...,n_k^{RT}$

\begin{IEEEeqnarray}{rcl}
\left[A^{RT}_{kk}\right]_{i,j} & := &  \int_{\widehat{\Omega}}  \left(\nabla \Bmull{\vect{p}+\vel{e}_k}{\vect{\alpha} +\vel{e}_k, i} \right)^T\ \mathfrak{C}^{RT}_k \ \nabla \Bmull{\vect{p}+\vel{e}_k}{\vect{\alpha} +\vel{e}_k, j} \;\d\boldsymbol\eta   + \sigma \left( \left(\tilde{J}_{\F}  \vel{e}_k\Bmull{\vect{p}+\vel{e}_k}{\vect{\alpha} +\vel{e}_k, i } \right) \circ\F^{-1} ,\left( \tilde{J}_{\F}  \vel{e}_k\Bmull{\vect{p}+\vel{e}_k}{\vect{\alpha} +\vel{e}_k, j } \right) \circ\F^{-1}\right)\nonumber \\  
& & +\int_{[0,1]^3} 2\nu \left\{ \left[ R_k\left(\nabla \Bmull{\vect{p}+\vel{e}_k}{\vect{\alpha} +\vel{e}_k, i }{J}_{\F}^{-1}\right)\right]^s:\left[(\mathbb{H}_{\F}\mathbf{e}_k\Bmull{\vect{p}+\vel{e}_k}{\vect{\alpha} +\vel{e}_k, j }){J}_{\F}^{-1}\right]^s \right. \nonumber \\
& &\quad + \left. \left[(\mathbb{H}_{\F}\mathbf{e}_k\Bmull{\vect{p}+\vel{e}_k}{\vect{\alpha} +\vel{e}_k, i }){J}_{\F}^{-1}\right]^s:\left[ R_k\left(\nabla \Bmull{\vect{p}+\vel{e}_k}{\vect{\alpha} +\vel{e}_k, j }{J}_{\F}^{-1}\right)\right]^s\right. \nonumber \\
& & \quad + \left. \left|\left| \left[(\mathbb{H}_{\F}\mathbf{e}_k\Bmull{\vect{p}+\vel{e}_k}{\vect{\alpha} +\vel{e}_k, j }){J}_{\F}^{-1}\right]^s\right|\right|^2_F \right\} |\mathrm{det}(J_{\F})|  \;\d\boldsymbol\eta, \label{eq:Akk_RT}
\end{IEEEeqnarray}

%\begin{IEEEeqnarray}{rcl}
%\left[A^{RT}_{kk}\right]_{i,j} & := &  \int_{[0,1]^3}  \left(\nabla \Bmull{\vect{p}+\vel{e}_k}{\vect{\alpha} +\vel{e}_k, i} \right)^T\ \mathfrak{Q}^{RT}_k \ \nabla \Bmull{\vect{p}+\vel{e}_k}{\vect{\alpha} +\vel{e}_k, j} \;\d\boldsymbol\eta \nonumber \\ 
% & &  + \theta \left( \left( {\left(\mathrm{det}(J_{\F})\right)}^{-1}J_{\F}  \vel{e}_k\Bmull{\vect{p}+\vel{e}_k}{\vect{\alpha} +\vel{e}_k, i } \right) \circ\F^{-1} ,\left( {\left(\mathrm{det}(J_{\F})\right)}^{-1}J_{\F}  \vel{e}_k\Bmull{\vect{p}+\vel{e}_k}{\vect{\alpha} +\vel{e}_k, j } \right) \circ\F^{-1}\right)\nonumber \\ \B
%& & +\int_{[0,1]^3} 2\nu \left\lbrace \Bmull{\vect{p}+\vel{e}_k}{\vect{\alpha} +\vel{e}_k, i} \Bmull{\vect{p}+\vel{e}_k}{\vect{\alpha} +\vel{e}_k, j} (\|\nabla^s R_k\|_{F})^2 + \Bmull{\vect{p}+\vel{e}_k}{\vect{\alpha} +\vel{e}_k, i}  \nabla^s R_k:\left[ R_k(\nabla \Bmull{\vect{p}+\vel{e}_k}{\vect{\alpha} +\vel{e}_k, j} )^T\right]^s \right. \nonumber \\
%& & \left. + \Bmull{\vect{p}+\vel{e}_k}{\vect{\alpha} +\vel{e}_k, j}  \nabla^s R_k:\left[ R_k(\nabla \Bmull{\vect{p}+\vel{e}_k}{\vect{\alpha} +\vel{e}_k, i} )^T\right]^s
%\right\rbrace\MM |\mathrm{det}(J_{\F})| \B \;\d\boldsymbol\eta, \label{eq:Akk_RT}
%\end{IEEEeqnarray}
where 
\begin{IEEEeqnarray}{c}
\mathfrak{C}_k^{RT}(\boldsymbol \eta) = \nu |\mathrm{det}(J_{\F})| J_{\F}^{-1}\left(\|R_k\|_2^2I+R_kR_k^T \right)J_{\F}^{-T},  \  \label{eq:Q_RT}
\end{IEEEeqnarray}

$\tilde{J}_{\F} := \left(\text{det}(J_{\F})\right)^{-1}
J_{\F}$, $R_k:= \tilde{J}_{\F}\ \vel{e}_k$ and $\mathbb{H}_{\F} $ is the (trivariate) Hessian tensor $\mathbb{H}_{\F} := \nabla \tilde{J}_{\F}$, with the convention that, for a given vector $\vect{w}\in\mathbb{R}^3$, 
$$ \mathbb{H}_{\F} \vect{w} := \begin{bmatrix} \left(\partial_{\eta_1}\tilde{J}_{\F}\right)\vect{w}, & \left(\partial_{\eta_2}\tilde{J}_{\F}\right)\vect{w}, & \left(\partial_{\eta_3}\tilde{J}_{\F}\right)\vect{w} \end{bmatrix}.$$ 

\B
 Here and
throughout, $\|\cdot\|_2$ denotes the Euclidean vector norm and  the
induced matrix norm,  $\|\cdot\|_F$ refers to the Frobenius matrix norm and  $[\ \cdot \ ]^s$ denotes the symmetric part.  
 Note that the last integral in \eqref{eq:Akk_RT} is zero when $\F$ is the identity map.

\section{Preconditioners for the whole system} \label{sec:SNSprec}

\newcommand{\precV}{P_V}
\newcommand{\precVc}[1]{P_{V,#1}}
\newcommand{\precQ}{P_Q}

In this section we introduce the preconditioning strategies that we consider in our numerical tests.
In what follows $\precV$ represents a preconditioning  matrix for the block $A$ and $\precQ$ a preconditioning matrix for $S$, where  

\begin{equation}
S = BA^{-1}B^T
\label{eq:schur}
\end{equation} 
is the (negative) Schur complement.

Once $\precV$ and $\precQ$ are constructed (this will be discussed in the next section), one can set up 
%block diagonal or block triangular preconditioners 
suitable preconditioners
to be used in the context of Krylov iterative methods \cite{Benzi2005,Elman2014,Wathen1993,Silvester1994}. 
We select three approaches.

{
In the first one, we consider the block diagonal preconditioner \cite{Murphy2000}
\begin{equation}
\mathcal{P}_{D}=\begin{bmatrix}
\precV & 0\\
0 & \precQ 
\end{bmatrix},
\label{eq:P_D}
\end{equation}
which, being symmetric and positive definite, preserves the symmetry of the problem. 
Therefore it can be coupled with a method for symmetric systems such as MINRES \cite{Paige1975}.
In the other two approaches, we respectively consider the block triangular \cite{Murphy2000} and constrained \cite{Keller2000} preconditioners  
\begin{equation}
\mathcal{P}_{T}=
\begin{bmatrix}
       \precV & B^T          \\[0.3em]
	0 & -\precQ
     \end{bmatrix} 
\label{eq:P_T}
\end{equation}
and
\begin{equation}
		\mathcal{P}_{C}=
     \begin{bmatrix}
       \precV & B^T     \\[0.3em]
	B & B\precV^{-1}B^T\!\!-\precQ
     \end{bmatrix},
\label{eq:P_C}
\end{equation}
both coupled with the GMRES method \cite{Saad1986}. We remark that $\mathcal{P}_{C}^{-1}$ can be applied efficiently thanks to the factorization
}
\begin{equation*}
\mathcal{P}_{C}^{-1}=
 \begin{bmatrix}
       I & - \precV^{-1}B^T         \\[0.3em]
	 0 & I
     \end{bmatrix}     
  \begin{bmatrix}
  I & 0 \\[0.3em]
  0 & -\precQ^{-1}
\end{bmatrix}     
\begin{bmatrix}
       I &0          \\[0.3em]
	-B & I
     \end{bmatrix}
     \begin{bmatrix}
       \precV^{-1} & 0          \\[0.3em]
	 0 & I
     \end{bmatrix},
\end{equation*}
where, here and throughout the paper, $I$ denotes the identity matrix of conforming order.

 \section{Preconditioners for $\precV$ and $\precQ$: 
   the simple choice}\label{sec:block-prec-simple-case}
\GS Our  choice for the preconditioning block $\precV$ has a   block diagonal structure:
\begin{equation}
\precV := \begin{bmatrix}
\precVc1 & 0 & 0\\
0 &  \precVc2 & 0\\
0 & 0 & \precVc3
\end{bmatrix};
\label{eq:precV}
\end{equation}
 the blocks $\precVc{k}$ are a  simplified version of to the blocks
 $A_{kk}$ where the geometry map and the kinematic viscosity are  replaced by the
 identity map and identity function, respectively. In other words, analogously to \eqref{eq:a_form} and \eqref{eq:theta}, we define in the parametric domain 
 $$
 \hat{a}(\hat{\vel{w}},\hat{\vel{v}}) := \int_{\widehat{\Omega}}2 \,\nabla^s\hat{\mathbf{w}}:\nabla ^s \hat{\mathbf{v}}\; \d \widehat{\Omega}, $$
 $$  \hat{\sigma}(\hat{\vel{w}}  , \hat{\vel{v}}  ):= \int_{\partial\widehat{\Omega}} 2
 \left[\frac{ C_{pen}}{h}\hat{\vel{w}}  \! \cdot\! \hat{\vel{v}}    - \left((\nabla^s\hat{\vel{w}} )\hat{\vel{n}}\right)\cdot\hat{\vel{v}}  -\left((\nabla^s\hat{\vel{v}} )\hat{\vel{n}}\right)\cdot\hat{\vel{w}} \right] \;\d\hat{\Gamma} ,
 $$
 where $\hat{\vel{n}}$ is the exterior normal to the boundary
 $\partial\widehat{\Omega}$.  Therefore for TH discretization,   according to \eqref{eq:Akk_TH}, for $i,j=1,...,n_{V,k}^{TH}$ and  $k=1,2,3$ we define
\begin{IEEEeqnarray}{l}
 \left[\precVc{k}^{TH}\right]_{i,j}   := \hat{a}(\vel{e}_k\Bmull{\vect{p}+1}{\vect{\alpha}, i} , \vel{e}_k \Bmull{\vect{p}+1}{\vect{\alpha}, j} )=
\int_{\widehat{\Omega}} \  \left(\nabla \Bmull{\vect{p}+1}{\vect{\alpha}, i} \right)^T   \, \mathfrak{T}_k\, \nabla \Bmull{\vect{p}+1}{\vect{\alpha} , j} \;\d\boldsymbol\eta, \label{eq:precVth}
\end{IEEEeqnarray}
where  $\mathfrak{T}_k  = I+\vel{e}_k\vel{e}_k^T$, while for RT discretization, according to \eqref{eq:Akk_RT},    for $i,j=1,...,n_{V,k}^{RT}$ and $k=1,2,3$  we define 
\begin{IEEEeqnarray}{rl}
\left[ \precVc{k}^{RT}\right]_{i,j}   :& = \hat{a}(\mathbf{e}_k\Bmull{\vect{p}+\vel{e}_k}{\vect{\alpha} +\vel{e}_k, i} , \mathbf{e}_k\Bmull{\vect{p}+\vel{e}_k}{\vect{\alpha} +\vel{e}_k, j} )+ \hat{\sigma}\left(\mathbf{e}_k\Bmull{\vect{p}+\vel{e}_k}{\vect{\alpha} +\vel{e}_k, i} ,\ \mathbf{e}_k\Bmull{\vect{p}+\vel{e}_k}{\vect{\alpha} +\vel{e}_k, j} \right)\nonumber \\
& = \int_{\widehat{\Omega}} \ \left( \nabla \Bmull{\vect{p}+\vel{e}_k}{\vect{\alpha} +\vel{e}_k, i} \right)^T\, \mathfrak{T}_k\, \nabla \Bmull{\vect{p}+\vel{e}_k}{\vect{\alpha} +\vel{e}_k, j} \;\d\boldsymbol\eta +\hat{\sigma}\left(\mathbf{e}_k\Bmull{\vect{p}+\vel{e}_k}{\vect{\alpha} +\vel{e}_k, i} ,\ \mathbf{e}_k\Bmull{\vect{p}+\vel{e}_k}{\vect{\alpha} +\vel{e}_k, j} \right).\nonumber \\
\label{eq:precVrt}
\end{IEEEeqnarray}
 \B

Exploiting the tensor product structure of the basis functions, we can write
\begin{IEEEeqnarray}{c}
\IEEEyesnumber\label{eq:prec3D_TH} \IEEEyessubnumber* 
\precVc{1}^{TH}\!  =  K_3^{TH}\!  \otimes\!  M_2^{TH}\!  \otimes\!  M_1^{TH}   +  M_3^{TH}\!  \otimes\!  K_2^{TH}\!  \otimes \!  M_1^{TH}  +  2 M_3^{TH}\!  \otimes \!  M_2^{TH}\!  \otimes\! K_1^{TH},  \\  
\precVc{2}^{TH}\!   =  K_3^{TH}\!  \otimes\!  M_2^{TH}\!  \otimes\!  M_1^{TH}   +  2M_3^{TH}\!  \otimes\!  K_2^{TH}\!  \otimes\!  M_1^{TH}  +  M_3^{TH}\!  \otimes\!  M_2^{TH}\!  \otimes\!  K_1^{TH},  \\ 
\precVc{3}^{TH}\!   =  2 K_3^{TH}\!  \otimes\!  M_2^{TH}\!  \otimes \! M_1^{TH}  +   M_3^{TH}\!  \otimes\!  K_2^{TH}\!  \otimes\!  M_1^{TH}   +  M_3^{TH}\!  \otimes\!  M_2^{TH}\!  \otimes\!  K_1^{TH},
\end{IEEEeqnarray}
and
\begin{IEEEeqnarray}{c}
\IEEEyesnumber\label{eq:prec3D_RT}  \IEEEyessubnumber*
\precVc{1}^{RT}\!   =  \widetilde {{K}}_3^{RT}\!  \otimes\!  \widetilde {{M}}_2^{RT}\!  \otimes\! M_1^{RT} + \widetilde {M}_3^{RT} \! \otimes\!  \widetilde {K}_2^{RT}\!  \otimes M_1^{RT} + 2 \widetilde{M}_3^{RT}\!  \otimes\!  \widetilde{M}_2^{RT}\!  \otimes\!  K_1^{RT},  \\ 
\precVc{2}^{RT} \!  =  \widetilde{K}_3^{RT}\!  \otimes\!  M_2^{RT}\!  \otimes\!  \widetilde{M}_1^{RT} + 2\widetilde{M}_3^{RT}\!  \otimes\!  K_2^{RT}\!  \otimes\!  \widetilde{M}_1^{RT} + \widetilde{M}_3^{RT}\! \otimes\!  M_2^{RT}\!  \otimes\!  \widetilde{K}_1^{RT},   \\ 
\precVc{3}^{RT}\!   =  2 K_3^{RT} \! \otimes\!  \widetilde{M}_2^{RT}\!  \otimes\!  \widetilde{M}_1^{RT} + M_3^{RT}\!  \otimes\!  \widetilde{K}_2^{RT}\! \otimes\!  \widetilde{M}_1^{RT} + M_3^{RT}\!  \otimes\!  \widetilde{M}_2^{RT}\!  \otimes \! \widetilde{K}_1^{RT},
\end{IEEEeqnarray} 
where for $k=1,2,3$ the univariate matrix factors are
\begin{IEEEeqnarray*}{l }
\left[K_k^{TH}\right]_{l,s} \!  = \!  \int_{[0,1]}\!\!\!\!\! (\bunivv{p+1}{\alpha_k,l})'(\eta_k) (\bunivv{p+1}{\alpha_k,s})'(\eta_k)\, \d\eta_k,   \nonumber \\
\left[M_k^{TH}\right]_{l,s} \!  = \! \int_{[0,1]}\!\!\!\!\! \bunivv{p+1}{\alpha_k,l}(\eta_k) \ \bunivv{p+1}{\alpha_k,s}(\eta_k)\, \d\eta_k,   \nonumber
\end{IEEEeqnarray*}
for $l,s=2,...,m_{\alpha_k}^{p+1}\! - \! 1$, and 
\begin{IEEEeqnarray*}{ll}
\left[K_{k}^{RT}\right]_{l,s} \!  = &\!  \int_{[0,1]}\!\!\!\!\! (\bunivv{p+1}{\alpha_k +1,l})'(\eta_k) (\bunivv{p+1}{\alpha_k+1,s})'(\eta_k)\, \d\eta_k,  \\
\left[M_{k}^{RT}\right]_{l,s} \!  = &\! \int_{[0,1]}\!\!\!\!\! \bunivv{p+1}{\alpha_k+1,l}(\eta_k) \ \bunivv{p+1}{\alpha_k+1,s}(\eta_k)\, \d\eta_k,    \nonumber\\
\end{IEEEeqnarray*}
for $l,s=2,...,\dimpu{k}\! - \! 1$, and finally 
\begin{IEEEeqnarray*}{ll}
 \left[\widetilde{K}_{k}^{RT}\right]_{l,s}  = &  \int_{[0,1]}\!\!\!\!\! (\bunivv{p}{\alpha_k,l})'(\eta_k) (\bunivv{p}{\alpha_k,s})'(\eta_k)\, \d\eta_k  \\& \quad -\bigg[ (\bunivv{p}{\alpha_k,l})'(1) \bunivv{p}{\alpha_k,s}(1) 
  - (\bunivv{p}{\alpha_k,l})'(0) \bunivv{p}{\alpha_k,s}(0)  + (\bunivv{p}{\alpha_k,s})'(1) \bunivv{p}{\alpha_k,l}(1) \\
 &\, \qquad - (\bunivv{p}{\alpha_k,s})'(0) \bunivv{p}{\alpha_k,l}(0)   -2\frac{C_{pen}}{h}\big(\bunivv{p}{\alpha_k,l}(1)\bunivv{p}{\alpha_k,s}(1)   + 
 \bunivv{p}{\alpha_k,l}(0)\bunivv{p}{\alpha_k,s}(0) \big) \bigg],  \nonumber \\
\left[\widetilde{M}_k^{RT}\right]_{l,s} =&  \int_{[0,1]}\!\!\!\!\! \bunivv{p}{\alpha_k,l}(\eta_k) \ \bunivv{p}{\alpha_k,s}(\eta_k)\, \d\eta_k
\end{IEEEeqnarray*}
for $l,s=1,...,\dimp{k}$.

 Now we consider the construction of $\precQ$. 
\GS The Schur complement $S$ is spectrally
equivalent to the (weighted) pressure mass matrix
\begin{equation}
  \label{eq:weighted-pressure-mass}
  \begin{aligned}
    &  \left[Q^{TH}\right]_{i,j} := \int_{\Omega}\nu^{-1}\rho_i^{TH}\rho_j^{TH}\;\d\Omega = \int_{\widehat{\Omega}}\nu^{-1}\Bmull{\vect{p}}{\vect{\alpha}, i}\Bmull{\vect{p}}{\vect{\alpha}, j}\ g^{TH}\;\d\vect{\eta}, \\
& \left[Q^{RT}\right]_{i,j} := \int_{\Omega}\nu^{-1}\rho_i^{RT}\rho_j^{RT}\;\d\Omega = \int_{\widehat{\Omega}}\nu^{-1}\Bmull{\vect{p}}{\vect{\alpha}, i}\Bmull{\vect{p}}{\vect{\alpha}, j}\ g^{RT}\;\d\vect{\eta},
  \end{aligned}
\end{equation}
for $i,j =1,...,n_{Q}$, where $g^{TH}(\vect{\eta}):= |\mathrm{det}(J_{\F})|$ and
$g^{RT}(\vect{\eta}):= |\mathrm{det}(J_{\F})|^{-1}$. The equivalence
holds uniformly with respect to a variable kinematic viscosity $\nu$, see
\cite{grinevich2009iterative}. \B
However, as for $\precV$, in our simple approach we drop the
dependence on $\nu$ and the geometry mapping,  by selecting: 
\begin{equation*}
\left[\precQ^{TH}\right]_{i,j} :=\left[\precQ^{RT}\right]_{i,j}  :=  \int_{\widehat{\Omega}} \Bmull{\vect{p}}{\vect{\alpha}, i}\Bmull{\vect{p}}{\vect{\alpha}, j} \; \d \,\vect{\eta} \quad i,j =1,...,n_Q;
\end{equation*}
as for \eqref{eq:precVth} and \eqref{eq:precVrt}.
Exploiting again the tensor product structure of the basis we can write $\precQ$ as
\begin{equation}
\precQ = M_3 \otimes M_2 \otimes M_1,
\label{eq:precQ}
\end{equation}
where for $k=1,2,3$ and for  $ l,s=1,...,n_Q$
\begin{equation*}
\left[M_k\right]_{l,s} =  \int_{[0,1]} \bunivv{p}{\alpha_k,l}(\eta_k) \  \bunivv{p}{\alpha_k,s}(\eta_k)\, \d\eta_k.
\end{equation*}

\subsection{Spectral properties}
\label{sec:spec_prop}

\GS
%It is known that the choice $\precV = A$ and $\precQ = Q$ into either \eqref{eq:P_D}, \eqref{eq:P_T} or \eqref{eq:P_C} leads to a robust preconditioning strategy. Therefore, a desirable requirement is that $\precV$ and $\precQ$ are spectrally equivalent to $A$ and $Q$, respectively. 
A desirable requirement for all the strategies proposed in Section \ref{sec:SNSprec} is that $\precV$ and $\precQ$ are spectrally equivalent to $A$ and $Q$, respectively.
We analyse here the spectral properties of $\precV^{-1} A$ and $\precQ^{-1} Q$. We refer to \cite[Section 4.2]{Elman2014}, where such properties are used to derive explicit bounds for the eigenvalues of the preconditioned system $\mathcal{P}^{-1} \mathcal{A}$, in the special case of the block diagonal preconditioner. In particular, if the eigenvalues of $\precV^{-1} A$ and $\precQ^{-1} Q$ are bounded away from 0 and infinity uniformly with respect to $h$ and $p$, then so are the eigenvalues of the full system. 
%A desirable requirement for all the strategies proposed in Section \ref{sec:SNSprec} is that $\precV$ and $\precQ$ are spectrally equivalent to $A$ and $Q$, respectively. For this reason, we analyse here the spectral properties of $\precV^{-1} A$ and $\precQ^{-1} Q$.
%{ We refer to \cite[Section 4.2]{Elman2014}, where such properties are used to derive explicit bounds for the eigenvalues of the preconditioned system $\mathcal{P}^{-1} \mathcal{A}$, in the special case of the block diagonal preconditioner.

The bilinear forms $a(\cdot,\cdot)$ and $\hat{a}(\cdot,\cdot)$ satisfy
\begin{IEEEeqnarray}{rcll}
 2 C_{\text{Korn}} \nu_{\min} \left|\vel{v}\right|^2_{\mathbf{H}^1(\Omega)}  \leq  & \, a(\vel{v},\vel{v}) \, & \leq  2 \nu_{\max}  \left|\vel{v}\right|^2_{\mathbf{H}^1(\Omega)} \qquad & \forall \, \mathbf{v} \, \in  \mathbf{H}^1_0(\Omega), \label{eq:a_bounds1}\\
  2  \hat {C}_{\text{Korn}}  \left|\hat{\vel{v}}\right|^2_{\mathbf{H}^1(\widehat{\Omega})}  \leq & \, \hat{a}(\hat{\vel{v}},\hat{\vel{v}}) \,& \leq  2   \left|\hat{\vel{v}}\right|^2_{\mathbf{H}^1(\widehat{\Omega})} \qquad & \forall \, \hat{\mathbf{v}} \, \in  \mathbf{H}^1_0(\widehat{\Omega}), \label{eq:a_bounds2}
  \end{IEEEeqnarray}
%% 2  {C}_{\text{Korn}} \hat{\nu}_{\min} \left|\hat{\vel{v}}\right|^2_{\mathbf{H}^1(\widehat{\Omega})}  \leq  \hat{a}(\hat{\vel{v}},\hat{\vel{v}}) \, \leq  2 \hat{\nu}_{\max}  \left|\hat{\vel{v}}\right|^2_{\mathbf{H}^1(\widehat{\Omega})} \qquad \forall \, \hat{\mathbf{v}} \, \in  \mathbf{H}^1_0(\widehat{\Omega}), \label{eq:a_bounds2}
where $\left|\ \cdot\ \right|_{\mathbf{H}^1(\cdot)}$ denotes the usual
$\mathbf{H}^1$-seminorm, $C_{\text{Korn}}$ and $\hat
{C}_{\text{Korn}}$ 
are the Korn constants (for homogeneous Dirichlet
boundary conditions on the whole boundary we have $C_{\text{Korn}} = \hat
{C}_{\text{Korn}}   = 1/2$, see  \cite[Section 6.3]{Ciarlet1988})  and 
$$ \nu_{\min} := \inf_{\Omega} \nu, \qquad \nu_{\max} := \sup_{\Omega} \nu.
$$
%$$ \qquad \hat{\nu}_{\min} := \inf_{\widehat{\Omega}} \hat{\nu}, \qquad \hat{\nu}_{\max} := \sup_{\widehat{\Omega}} \hat{\nu}.$$
%
We  also have that the bilinear forms $a(\cdot, \cdot) + \sigma(\cdot, \cdot)$ and $\hat{a}(\cdot, \cdot) + \hat{\sigma}(\cdot, \cdot)$ in the discrete spaces satisfy
\begin{align}  C_{1} \| \vel{v}_h  \|^2_{\mathbf{H}_{pen}^1( {\Omega})} \leq a( {\vel{v}}_h, {\vel{v}}_h) + \sigma( {\vel{v}}_h, {\vel{v}}_h) \leq  C_{2}\| \vel{v}_h \|^2_{\mathbf{H}_{pen}^1( {\Omega})} \qquad \forall \, \mathbf{v}_h \, \in   V^{RT}_{h,0}, \label{eq:RT-weak-Korn} \\
 \hat{C}_{1}\| \hat{\vel{v}}_h  \|^2_{\mathbf{H}_{pen}^1( \widehat{\Omega})} \leq \hat{a}( \hat{{\vel{v}}}_h, \hat{{\vel{v}}}_h) + \hat{\sigma}( \hat{{\vel{v}}}_h, {\hat{\vel{v}}}_h) \leq \hat{C}_{2}\| \hat{\vel{v}}_h \|^2_{\mathbf{H}_{pen}^1( \widehat{\Omega})} \qquad \forall \, \hat{\mathbf{v}}_h \, \in   \hat{V}^{RT}_{h,0},  \label{eq:RT-weak-Korn-parametric}
  \end{align}
  where the norm $\| \cdot \|_{\mathbf{H}_{pen}^1( \widehat{\Omega})} $ is defined as
   $
  \|\cdot \|^2_{\mathbf{H}_{pen}^1( {\Omega})}:=\| \cdot\|^2_{\mathbf{H}^1( {\Omega})} + \frac{C_{pen}}{h}\| \cdot \|^2_{L^2({\partial\Omega})}
  $
  and $C_1$, $C_2$, $\hat{C}_1$ and $\hat{C}_2$ are constants depending on 
  $C_{pen}$ and on the inverse estimate constants of the discrete
  spaces   $ V^{RT}_{h,0}$ and  $\hat{V}^{RT}_{h,0}$ respectively:
  these inequalities follows from   \cite[Lemma 6.2]{Evans2013},  \cite[Lemma 6.3]{Evans2013},\cite[Eq. (6.9)]{Evans2013} and the equivalence between   $\parallel \cdot \parallel_{\mathbf{H}_{pen}^1(  {\Omega})}$ and 
    $
  \left| \  \cdot \ \right|^2_{\mathbf{H}_{pen}^1( {\Omega})}:= \left|\ \cdot\ \right| ^2_{\mathbf{H}^1( {\Omega})} + \frac{C_{pen}}{h}\parallel \cdot \parallel^2_{L^2({\partial\Omega})}  $.

We start by proving bounds on the eigenvalues of $\precV^{-1} A$.
 
\begin{thm}
\label{prop:1}
It holds 
\begin{equation}
 \delta \leq \lambda_{\min} \left(\precV^{-1} A \right), \qquad \lambda_{\max} \left(\precV^{-1} A\right) \leq \Delta,
\label{eq:specV}
\end{equation}
where $\delta$ and $\Delta$ are positive constants that do not depend on $h$ or on $p$.
% where
%$$
%\delta^{TH} :=\delta^{TH}\left(J_{\F},   C_{\mathrm{Korn}}, \nu_{\min},  \right) , \quad
%\Delta^{TH}:= \Delta^{TH} \left(J_{\F}, \nu_{\max}, \right) , 
%$$
%$$
%\delta^{RT} :=\delta^{RT}\left(\tilde{J}_{\F}, \mathbb{H} , C_1, \hat{C}_2\right) , \quad
%\Delta^{RT}:= \Delta^{RT} \left( \tilde{J}_{\F},  \mathbb{H} , C_2, \hat{C}_1\right),
%$$
%with $\tilde{J}_{\F}:= \left| \det(J_{\F})\right|^{-1}J_{\F}$ and $\mathbb{H}:= \nabla \tilde{J}_{\F}$.
\end{thm}

\begin{proof}
We begin with TH discretization case, proving \eqref{eq:specV} for $
\delta = \delta^{TH}$ and $ \Delta = \Delta^{TH}$. 
Let $\hat{\vel{v}}_h \in \hat{V}^{TH}_{h,0}$ and let ${\vel{v}}_h := \hat{\vel{v}}_h \circ \F^{-1} \in {V}^{TH}_{h,0}$. 
Moreover, let $\vect{v}$ be the coordinate vector of $\hat{\vel{v}}_h$ with respect to the basis \eqref{eq:basisRT_V1}. 
By the Courant-Fischer theorem, \eqref{eq:specV}  is equivalent to find  $\delta^{TH}$ and $\Delta^{TH}$ such that
$$
\delta^{TH}\leq \frac{\vect{v}^TA^{TH}\vect{v}}{\vect{v}^T\precV^{TH}\vect{v}}  \leq \Delta^{TH} . $$
%
%\frac{a(\vel{v}_h, \vel{v}_h)}{\hat{a}(\hat{\vel{v}}_h, \hat{\vel{v}}_h)}
%
Using \eqref{eq:a_bounds1}, we have $$2 C_{\text{Korn}} \nu_{\min}
\left|\vel{v_h}\right|^2_{\mathbf{H}^1(\Omega)}  \leq   \,
\vect{v}^TA^{TH}\vect{v}  \, \leq  2 \nu_{\max}
\left|\vel{v_h}\right|^2_{\mathbf{H}^1(\Omega)}.$$ Using
\eqref{eq:a_bounds2} and decomposing $ \hat{\vel{v}}_h  =
\hat{\vel{v}}_{h,1} + \hat{\vel{v}}_{h,2} + \hat{\vel{v}}_{h,3}  $,
where $\hat{\vel{v}}_{h,k}$ are the cartesian components of
$\hat{\vel{v}}_{h}$, we have for $k=1,2,3$,
$$  2  \hat {C}_{\text{Korn}}  \left|\hat{\vel{v}}_{h,k}\right|^2_{\mathbf{H}^1(\widehat{\Omega})}  \leq  \, \hat{a}(\hat{\vel{v}}_{h,k},\hat{\vel{v}}_{h,k}) \, \leq  2   \left|\hat{\vel{v}}_{h,k}\right|^2_{\mathbf{H}^1(\widehat{\Omega})};
$$
summing the three bounds above and using $
\hat{a}(\hat{\vel{v}}_{h,1},\hat{\vel{v}}_{h,1})
+\hat{a}(\hat{\vel{v}}_{h,2},\hat{\vel{v}}_{h,2})
+\hat{a}(\hat{\vel{v}}_{h,3},\hat{\vel{v}}_{h,3}) =
\vect{v}^T\precV^{TH}\vect{v} $ yields 
$$  2  \hat {C}_{\text{Korn}}
\left|\hat{\vel{v}}_{h}\right|^2_{\mathbf{H}^1(\widehat{\Omega})}  \leq 
\, \vect{v}^T\precV^{TH}\vect{v} \, \leq  2   \left|\hat{\vel{v}}_{h}\right|^2_{\mathbf{H}^1(\widehat{\Omega})};
$$
in conclusion it suffices to prove 
\begin{equation}
  \label{eq:TH-CF-equivalence} 
\frac{\delta^{TH}}{C_{\text{Korn}}{\nu}_{\min}} \leq \frac{\left|
    \vel{v}_h \right|^2_{\mathbf{H}^1(\Omega)}}{\left|\hat{\vel{v}}_h
  \right|_{\mathbf{H}^1(\widehat{\Omega})}^2 } \leq\frac{\hat C_{\text{Korn}}\Delta^{TH}}{\nu_{\max}},
\end{equation}
for suitable  $\delta^{TH}$ and $\Delta^{TH}$ and for all all $\hat{\vel{v}}_h\in \hat{V}_{h,0}^{TH}$ with $ {\vel{v}}_h =:
\hat{\vel{v}}_h\circ\F^{-1} \in {V}_{h,0}^{TH}$. In other words, we just need to prove the equivalence between $\left|\vel{v}_h\right|_{\mathbf{H}^1(\Omega)}$ and $\left|\hat{\vel{v}}_h\right|_{\mathbf{H}^1(\widehat{\Omega})}$.
One of the two  bounds is 
\begin{IEEEeqnarray}{ll}
 \left|\vel{v}_h \right|^2_{\mathbf{H}^1(\Omega)} & = \int_{\Omega} \left\| \nabla \vel{v}_h \right\|_{F}^2 \; \d \Omega = \int_{\hat \Omega}  \left|\det\left(J_{\F}\right)\right|
\left\| \nabla \hat{\vel{v}}_h  J_{\F}^{-1} \right\|_F^2
%\left( J_{\F}^{-T} \nabla \hat{\vel{v}}_h \right) : \left( J_{\F}^{-T} \nabla \hat{\vel{v}}_h \right)
\; \d\boldsymbol\eta \nonumber \\
\label{eq:TH_bound1}
& \leq \sup_{\hat \Omega}
%\lambda_{\max}\left(J_{\F}^{-1}J_{\F}^{-T}|\mathrm{det}J_{\F}|\right) 
\left\{\left|\det\left(J_{\F}\right)\right| \left\|J_{\F}^{-1}\right\|^2_2\right\}
\int_{\hat \Omega} \left\| \nabla \hat{\vel{v}}_h \right\|^2_F \; \d\boldsymbol\eta   = \sup_{\hat \Omega}
%\lambda_{\max}\left(J_{\F}^{-1}J_{\F}^{-T}|\mathrm{det}J_{\F}|\right)
\left\{\left|\det\left(J_{\F}\right)\right| \left\|J_{\F}^{-1}\right\|^2_2\right\}
\left|\hat{\vel{v}}_h \right|^2_{\mathbf{H}^1(\widehat{\Omega})}    , \nonumber \\
\end{IEEEeqnarray}
where we used the fact that, given any two matrices $X,Y$ with
conforming dimensions, it holds $\left\| X Y\right\|_F^2 \leq
\left\|X\right\|^2_F \left\|Y \right\|^2_2$. For the other bound, just
observe that   $\hat{\vel{v}}_h :=  {\vel{v}}_h\circ\F$, and then
\begin{equation}
\label{eq:TH_bound2}
 \left|\hat{\vel{v}}_h \right|^2_{\mathbf{H}^1(\widehat{\Omega})} \leq \sup_{  \Omega}
%\lambda_{\max}\left(J_{{\F}^{-1}}^{-1} J_{{\F}^{-1}}^{ -T}|\mathrm{det}J_{{\F}^{-1}}|\right) 
\left\{\left|\det\left(J_{\F^{-1}}\right)\right| \left\|J_{\F^{-1}}^{-1}\right\|^2_2\right\} 
\left|\vel{v}_h \right|^2_{\mathbf{H}^1(\Omega)}  =  \sup_{  \hat \Omega} 
\left\{ \frac{ \left\|J_{\F}\right\|^2_2 }{ \left|\det\left(J_{\F}\right)\right| } \right\} 
\left|\vel{v}_h \right|^2_{\mathbf{H}^1(\Omega)};
\end{equation}
This conclude the proof for the TH case. 

The  RT case is similar, we just highlight the differences. 
As above, from \eqref{eq:RT-weak-Korn} and
\eqref{eq:RT-weak-Korn-parametric}, we get 
\begin{equation}
  \label{eq:RT-equiv}
C_{1} \| \vel{v}_h  \|^2_{\mathbf{H}_{pen}^1( {\Omega})} \leq \vect{v}^TA^{RT}\vect{v}\leq  C_{2}\| \vel{v}_h \|^2_{\mathbf{H}_{pen}^1( {\Omega})},
\end{equation}

\begin{equation}
  \label{eq:RT-equiv-parametric}
  \hat{C}_{1}\| \hat{\vel{v}}_h  \|^2_{\mathbf{H}_{pen}^1(
    \widehat{\Omega})} \leq \vect{v}^T\precV^{RT}\vect{v} \leq \hat{C}_{2}\| \hat{\vel{v}}_h \|^2_{\mathbf{H}_{pen}^1( \widehat{\Omega})},
\end{equation}
where $\hat{\vel{v}}_h\in \hat{V}_{h,0}^{RT}$, $
{\vel{v}}_h=(( \text{det}(J_{\F}))^{-1}
J_{\F}\hat{\vel{v}}_h)\circ\F^{-1}=(\tilde{J}_{\F}\hat{\vel{v}}_h)\circ\F^{-1}
\in {V}_{h,0}^{RT}$ and  $\vect{v}$ is the common coordinate
vector. % and $\tilde{J}_{\F} = \left|\text{det}(J_{\F})\right|^{-1}J_{\F}$.  
Then, we look for  $\delta^{RT}$ and $\Delta^{RT}$ such that

$$
\delta^{RT}\frac{\hat{C}_2 }{C_1 } \leq \frac{\left\| \vel{v}_h \right\| ^2_{\mathbf{H}_{pen}^1(\Omega)}}{\left\|\hat{\vel{v}}_h  \right\|_{\mathbf{H}_{pen}^1(\widehat{\Omega})}^2 } \leq\frac{\hat{C}_1 }{C_2 } \Delta^{RT}.
$$
% i.e. we want to prove the equivalence between $\left\|\vel{v}_h\right\|_{\mathbf{H}_{pen}^1(\Omega)}$ and $\left\|\hat{\vel{v}}_h\right\|_{\mathbf{H}_{pen}^1(\widehat{\Omega})}$.
%We define the (trivariate) Hessian tensor $\mathbb{H}_{\F} := \nabla \tilde{J}_{\F}$.
 Direct computation shows that $ \nabla \left( \tilde{J}_{\F}\hat{\vel{v}}_h \right) = \tilde{J}_{\F}\nabla\hat{\vel{v}}_h + \mathbb{H}_{\F}\hat{\vel{v}}_h $, where $\tilde{J}_{\F}$ and $\mathbb{H}_{\F}$ as in Section \ref{sec:Stokes}.
It holds
\begin{IEEEeqnarray*}{lcl}
\left| \vel{v}_h\right|^2_{\mathbf{H}^1(\Omega)} & = & \int_{\Omega} \left\| \nabla \vel{v}_h \right\|^2_F \; \d \Omega = 
%\int_{\widehat{\Omega}}\left| \det(J_{\F}) \right| \nabla \left( \tilde{J}_{\F}\hat{\vel{v}}_h \right):\nabla\left(\tilde{J}_{\F}\hat{\vel{v}}_h\right) \d\widehat{\Omega}\nonumber \\
\int_{\widehat{\Omega}} \left| \det(J_{\F})\right| \left\| \nabla \left( \tilde{J}_{\F}\hat{\vel{v}}_h \right) J_{\F}^{-1} \right\|_F^2 \d\widehat{\Omega}\nonumber \\
& \leq &\ 2 \int_{\widehat{\Omega}}\left| \det(J_{\F}) \right|\left( \left\| \tilde{J}_{\F} \nabla\hat{\vel{v}}_h J_{\F}^{-1} \right\|_F^2 + \left\| \left(\mathbb{H}_{\F}\hat{\vel{v}}_h\right) J_{\F}^{-1} \right\|_F^2 \right)\d\widehat{\Omega} \nonumber \\
& \leq &\ 2 \sup_{\widehat{\Omega}} \left\{ \left| \det(J_{\F})\right|\left\| J_{\F}^{-1}\right\|_2^2  \left\|\tilde{J}_{\F}\right\|^2_{2}  , \left| \det(J_{\F})\right| \left\|J_{\F}^{-1}\right\|_2^2 \left\| \mathbb{H}_{\F}\right\|^2_{F} \right\} \left\|\hat{\vel{v}}_h\right\|^2_{\mathbf{H}^1(\widehat{\Omega})},\\
\end{IEEEeqnarray*}
where $\left\| \mathbb{H}_{\F}\right\|^2_{F}$ is the Frobenius tensor norm of $\mathbb{H}_{\F}$. Moreover, it holds
\begin{IEEEeqnarray*}{lcl}
\left\| \vel{v}_h\right\|^2_{L^2(\Omega)} & = & \int_{\Omega}|\vel{v}_h|^2\d\Omega = \int_{\widehat{\Omega}}\left|\det(J_{\F}) \right| 
\left\|\tilde{J}_{\F}\hat{\vel{v}}_h \right\|_2^2
%\left(\tilde{J}_{\F}\hat{\vel{v}}_h\right):\left(\tilde{J}_{\F}\hat{\vel{v}}_h\right)
\d\widehat{\Omega}  \leq  \ \sup_{\widehat{\Omega}} \left\{ \left|\det(J_{\F}) \right| \left\|\tilde{J}_{\F}\right\|^2_{2} \right\} \left\|\hat{\vel{v}}_h \right\|^2_{L^2(\widehat{\Omega})} ,\\
\left\| \vel{v}_h\right\|^2_{L^2(\partial\Omega)} & = & \int_{\partial\Omega}|\vel{v}_h|^2\d\Gamma \leq \left\|\text{cof}(\nabla\F)\right\|_{L^{\infty}(\widehat{\Omega}),l} \int_{\partial\widehat{\Omega}}
\left\| \tilde{J}_{\F}\hat{\vel{v}}_h \right\|_2^2
%\left(\tilde{J}_{\F}\hat{\vel{v}}_h\right):\left(\tilde{J}_{\F}\hat{\vel{v}}_h\right)
\d\hat{\Gamma} \\
& \leq & \ \left\|\text{cof}(\nabla\F)\right\|_{L^{\infty}(\widehat{\Omega}),l}\sup_{\partial\widehat{\Omega}} \left\{ \left\|\tilde{J}_{\F}\right\|^2_{2} \right\} \left\|\hat{\vel{v}}_h\right\|^2_{L^2(\partial\widehat{\Omega})} 
\end{IEEEeqnarray*}
where $\text{cof}(\cdot)$ refers to the matrix of the cofactors and $\left\|\cdot \right\|_{L^{\infty}(\widehat{\Omega}),l}$ is defined as in \cite{Evans2013trace}. 

By observing that
%Similarly, by using that $ \hat{\vel{v}}_h:= \tilde{J}_{\F}^{-1}({\vel{v}}_h\circ\F) $, it holds 
$ \hat{\vel{v}}_h = (\tilde{J}_{\F^{-1}} \vel{v}_h)\circ\F$, we can use similar argument to show that 
\begin{IEEEeqnarray*}{rcl}
\left| \hat{\vel{v}}_h\right|^2_{\mathbf{H}^1(\widehat{\Omega})} \, & \leq & \, 2 \sup_{{\Omega}} \left\{ \left| \det(J_{\F^{-1}})\right| \left\|J_{\F^{-1}}^{-1}\right\|_2^2 \left\| \tilde{J}_{\F^{-1}}\right\|^2_{2},  \left| \det(J_{\F^{-1}})\right| \left\|J_{\F^{-1}}^{-1}\right\|_2^2 \left\| \mathbb{H}_{\F^{-1}}\right\|^2_{F} \right\} \left\|{\vel{v}}_h\right\|^2_{\mathbf{H}^1({\Omega)}},\\
\left\| \hat{\vel{v}}_h\right\|^2_{L^2(\widehat{\Omega})} \, & \leq & \, \sup_{{\Omega}} \left\{ \left|\det(J_{\F^{-1}}) \right| \left\|\tilde{J}_{\F^{-1}}\right\|^2_{2} \right\} \left\|{\vel{v}}_h \right\|^2_{L^2({\Omega})} ,\\
\left\| \hat{\vel{v}}_h\right\|^2_{L^2(\partial\widehat{\Omega})} \, &  \leq & \, \left\|\text{cof}(\nabla\F^{-1})\right\|_{L^{\infty}({\Omega}),l}\sup_{\partial{\Omega}} \left\{ \left\|\tilde{J}_{\F^{-1}}\right\|^2_2 \right\} \left\|{\vel{v}}_h \right\|^2_{L^2(\partial{\Omega})}.
\end{IEEEeqnarray*}
This concludes the analysis of the RT case.

\end{proof}
\B

We next analyse $\precQ^{-1} Q$.

\begin{thm}
\label{prop:2}
It holds 
\begin{equation}
\theta \geq \lambda_{\min} \left( \precQ ^{-1} Q \right), \qquad \lambda_{\max} \left( \precQ ^{-1} Q \right) \leq \Theta,
\label{eq:specQ}
\end{equation}
 where $\theta$ and $\Theta$ are positive constants that do not depend on $h$ or on $p$.
% where
%$$
%\delta^{TH} :=\delta^{TH}\left(J_{\F},   C_{\mathrm{Korn}}, \nu_{\min},  \right) , \quad
%\Delta^{TH}:= \Delta^{TH} \left(J_{\F}, \nu_{\max}, \right) , 
%$$
%$$
%\delta^{RT} :=\delta^{RT}\left(\tilde{J}_{\F}, \mathbb{H} , C_1, \hat{C}_2\right) , \quad
%\Delta^{RT}:= \Delta^{RT} \left( \tilde{J}_{\F},  \mathbb{H} , C_2, \hat{C}_1\right),
%$$
%with $\tilde{J}_{\F}:= \left| \det(J_{\F})\right|^{-1}J_{\F}$ and $\mathbb{H}:= \nabla \tilde{J}_{\F}$.
\end{thm}
\B
\begin{proof}
We report the proof for TH discretization. The proof for the RT discretization can be derived in a analogous way.

By Courant-Fischer theorem, we need to prove
\begin{equation*}
\theta\leq
  \frac{\langle Q\vect{g},\vect{g}\rangle}{\langle \precQ \vect{g},\vect{g}\rangle}\leq \Theta \quad \forall \, \vect{g}\in \mathbb{R}^{n_Q}.
\end{equation*}
Let $\vect{g}\in\mathbb{R}^{n_Q}$ and $g_h =
\sum_{i=1}^{n_Q}[\vect{g}]_i\Bmull{\vect{p}}{\vect{\alpha},i}$. It
holds
\begin{equation}
  \label{eq:PQ}
  \vect{g}^T  Q^{TH} \vect{g} \      = \ \int_{\widehat{\Omega}} g_h^2   \nu^{-1}\B |\mathrm{det}(J_{\mathbf{\F}})|\;\d\widehat{\Omega}\leq 
\sup_{\widehat{\Omega}}\left(|\mathrm{det}(J_{\mathbf{\F}})|\  \nu^{-1}  \right) \int_{\widehat{\Omega}} g_h^2 \;\d\widehat{\Omega}   \leq \ \sup_{\widehat{\Omega}}\left(|\mathrm{det}(J_{\mathbf{\F}})|\   \nu^{-1}  \right) \vect{g}^T \precQ^{TH} \vect{g},
\end{equation}
and, in an analogous way, one can prove the other side of the inequality.

\end{proof}

\begin{remark}
The constants $\delta$, $\Delta$, $\theta$ and  $\Theta$ depend on the parametrization $\F$
and on the kinematic viscosity $\nu$. This dependence can be
inferred from the proof of Theorems \ref{prop:1}--\ref{prop:2}. Considering for
example the  TH case,  from \eqref{eq:TH-CF-equivalence}--\eqref{eq:TH_bound2} and using 
$$
\left[ \sup_{  \hat \Omega} \left\{ \frac{ \left\|J_{\F}\right\|^2_2 }{ \left|\det\left(J_{\F}\right)\right| } \right\}  \right]^{-1} = \inf_{  \widehat{\Omega}} \left\{ \frac{\left|\det\left(J_{\F}\right)\right|}{\left\|J_{\F}\right\|^2_2} \right\},
$$
we get to the following admissible choices 
$$
\delta^{TH}\!\!  = {C_{\mathrm{Korn}} \nu_{\min}} \inf_{\widehat{\Omega}} 
%\lambda_{\min}\left(J_{\F}^{-1}J_{\F}^{-T}|\mathrm{det}J_{\F}|\right), 
\left\{
  \frac{\left|\det\left(J_{\F}\right)\right|}{\left\|J_{\F}\right\|^2_2}
\right\}  \text{ and }
\Delta^{TH}\!\!  =  {\hat C ^{-1}_{\mathrm{Korn}}} \nu_{\max} \sup_{\widehat{\Omega}}
\left\{ \left|\det\left(J_{\F}\right)\right| \left\|J_{\F}\right\|^2_2 \right\}.
%\lambda_{\max}\left(J_{\F}^{-1}J_{\F}^{-T}|\mathrm{det}J_{\F}|\right).
$$

In a similar way, from \eqref{eq:PQ}, we have following admissible choices 
\begin{IEEEeqnarray*}{c'c}
   \theta^{TH}= \inf_{\widehat{\Omega}}(|\mathrm{det}(J_{\mathbf{\F}})| 
   \nu^{-1} ),\  \text{ and } \ \Theta^{TH}=  \sup_{\widehat{\Omega}}(|\mathrm{det}(J_{\mathbf{\F}})| \nu^{-1} ).
\end{IEEEeqnarray*}
\end{remark}

%The spectral bounds \eqref{eq:specV} and \eqref{eq:specQ} give a measure of the quality of our preconditioning strategy. In the special case of the block diagonal preconditioner and TH discretization, we now report explicit bounds for the eigenvalues of the preconditioned system, that can be proven using standard arguments.
%\begin{thm}
%The nonzero eigenvalues of $\left(\mathcal{P}_D^{TH}\right)^{-1} \mathcal{A}^{TH}$ are contained in
%$$ \left[ - \Theta^{TH}\Gamma^2 , \; \frac{1}{2}\left(\delta^{TH} - \sqrt{ \left(\delta^{TH}\right)^2 + 4 \gamma^2 \delta^{TH} \theta^{TH}} \right) \right] \; \cup \; \left[ \delta^{TH}, \; \Delta^{TH} + \Gamma^2 \Theta^{TH} \right], $$
%where the constants $\delta^{TH}, \Delta^{TH}, \theta^{TH}, \Theta^{TH}$, are defined in Propositions \ref{prop:1} and \ref{prop:2}, and 
%$$ \gamma := \inf \sup , \qquad \Gamma :=  \sup \sup .$$
%\end{thm}
%\begin{proof}
%It follows from the arguments used in \cite[Section 4.2]{Elman2014}.
%\end{proof}

\subsection{Preconditioners application: FD method}\label{sec:prec_app}
At each iteration of our iterative solver we have to solve
\begin{equation}
\mathcal{P} s = r,
\label{eq:prec_appl}
\end{equation}
where $r$ is the current residual and $\mathcal{P}$ is a 
{ preconditioner, that can be either matrix from \eqref{eq:P_D}, \eqref{eq:P_T} and \eqref{eq:P_C}.}
%preconditioner matrix, that is \eqref{eq:P_SPD}, \eqref{eq:P_NST} or \eqref{eq:P_NSC}.
Besides multiplications by $B$ or $B^T$, to accomplish this task we need to solve the linear systems with matrices $\precV $ and $\precQ$.
Thanks to \eqref{eq:kron_vec_sol} and the band structure of the univariate factors in \eqref{eq:precQ}, the solution of a linear system with matrix $\precQ$ is obtained in a direct way with only $O(pn_Q)$ FLOPs.

On the other hand, the solution of a linear system with matrix $\precV$ requires  to solve three Sylvester-like equations, one for each diagonal block $\precVc{k}$.
Following  \cite{Sangalli2016}, to accomplish this aim  we use the Fast Diagonalization (FD) direct method of  \cite{Lynch1964} 
and \cite{Deville2002}.  
 We now briefly explain its main features.

Consider the general  Sylvester-like system:
\begin{equation}
{R}q := (K_3 \otimes M_2 \otimes M_1 + M_3 \otimes K_2 \otimes M_1 +  M_3 \otimes M_2 \otimes K_1)\  q = t,
\label{eq:sylv_eq}
\end{equation}
with both $M_i$ and $K_i$ symmetric and positive definite matrices for $i=1,2,3$.
Let 
\begin{equation}
K_i U_i = M_iU_i D_i, \quad i=1,2,3,
\label{eq:eig_dec}
\end{equation}
be the eigendecomposition of the pencils $(K_i, M_i)$,
where $D_i$ are diagonal matrices containing the eigenvalues of $M_i^{-1}K_i$ and $U_i^{T}M_iU_i = I$.
 We have $M_i = U_i^{-T}U_i^{-1}$ and $K_i = U_i^{-T}D_i U_i ^{-1}$. 
 Then, we can factorize ${R}$ as
$$
{R}=(U_3\otimes U_2 \otimes U_1)^{-T}(I\otimes I\otimes D_1 + I\otimes D_2 \otimes I + D_3 \otimes I\otimes I)(U_3\otimes U_2 \otimes U_1)^{-1}\!.
$$
Exploiting \eqref{eq:krontranspose},  \eqref{eq:kronvec} and the factorization above, the solution of \eqref{eq:sylv_eq} can be computed by the following algorithm.

\begin{algorithm}[H]
\caption{ 3D FD method }\label{al:3Ddirect_P}
\begin{algorithmic}[1]
\State Compute the generalized eigendecompositions \eqref{eq:eig_dec}
\State Compute $\tilde{t} = (\U1 \otimes \U2 \otimes \U3)^T t $
\State Compute $\tilde{q} = \left(I\otimes I\otimes D_1 + I\otimes D_2 \otimes I +  D_3 \otimes I\otimes I\right)^{-1} \tilde{t} $
\State Compute $q = (\U1 \otimes \U2 \otimes \U3)\ \tilde{q} $
\end{algorithmic}
\end{algorithm}
Assuming for simplicity that the matrices $K_i$ and $M_i$ all have the same order $n$,
%Assuming for simplicity that $n = \mathrm{dim}(M_i) = \mathrm{dim}(K_i)$ and then $n^3 = n_{dof} = \mathrm{dim}(R)$, 
Algorithm \ref{al:3Ddirect_P} requires $12n^4+O(n^3)= 12 n_{dof}^{4/3}+O(n_{dof})$ FLOPs, where $n_{dof} = n^3$ denotes the order of $R$.
Step 1 and step 3 are optimal as they require only $O(n_{dof})$ FLOPs.
The asymptotic dominant cost, i.e. $ 12 n_{dof}^{4/3}$ FLOPs, is related to the  matrix-matrix products of step 2 and step 4, while step 1 and step 3 are optimal as they require only $O(n_{dof})$ FLOPs. 
{\GS
However step 2 and step 4, being BLAS level 3 operations, are typically implemented in a highly efficient way on modern computers. 
As a consequence, despite their superlinear computational cost, in practice they do not dominate the computational time of the overall iterative strategy {(see the numerical experiments of \cite{Sangalli2016} and the ones in the present paper for more details on this important point)}. }

 \section{Preconditioners for $\precV$ and $\precQ$   
  including coefficients information}
\label{sec:inc_geo}
 The proposed preconditioners $\precV$ and $\precQ$  from Section  
 \ref{sec:block-prec-simple-case}  are robust with respect to the
 mesh size and spline degree. However they do not incorporate any
 information from the coefficients (either the  geometry map  $\F$
 and or the kinematic viscosity $\nu$)  and in fact  the preconditioner's quality  is
 affected from the coefficients. This is reflected in the
 analysis of Section \ref{sec:spec_prop} (see Remark 1 for the TH
 case).  Numerical
 tests of Section \ref{sec:num_res} confirm this expectation. 
We therefore  present two  strategies that partially incorporate
$\nu$ and $\F$ in $\precV$ and $\precQ$, without increasing  the
preconditioners  computational
cost. \GS

First, we consider  a diagonal scaling. In particular, 
we replace $\precQ$  by $\precQg:= D_Q^{1/2} \precQ D_Q^{1/2}$,  where
$D_Q$
is a  diagonal matrices having diagonal entries  $\left[D_Q\right]_{i,i} =
\left[Q\right]_{i,i}/\left[\precQ\right]_{i,i} $. Even though we postpone a mathematical analysis
of it to a further work, the numerical tests in  Section \ref{sec:num_res} show that this
cheap modification of the preconditioner 
is sufficient to give $\precQg$ robustness with respect to the coefficients (not only $\F$, as
indicated, but also $\nu$). 

The same idea,  applied to $\precV$, while  able to incorporate
efficiently the contribution  of the scalar coefficient $\nu$,   is less effective when the
geometry parametrization  is far from a scaled identity. In this case
we propose to incorporate  some components of the geometry
parametrization into the univariate matrix factors appearing in
\eqref{eq:prec3D_TH} and \eqref{eq:prec3D_RT} (see the Appendix for
details) in order to build a preconditioner  $\precVgg $ such that
Algorithm \ref{al:3Ddirect_P} can still be used. Then, we  apply a
diagonal scaling. This leads to an effective preconditioner having the
form $
\precVg:=D_V^{1/2} \precVgg D_V^{1/2} $, where $D_V$ has diagonal
entries $\left[D_V\right]_{i,i} = \left[A\right]_{i,i}/\left[\precVgg\right]_{i,i}$.

We use the following notation:    $\geo{D} $, $\geo{T} $ and
$\geo{C}$  are the preconditioner matrices for the Stokes system  obtained by replacing
$\precV$ and $\precQ$ with $\precVg$ and $\precQg$  in
%\eqref{eq:P_SPD}, \eqref{eq:P_NST} 
\eqref{eq:P_D}, \eqref{eq:P_T}
and  \eqref{eq:P_C},
respectively. The corresponding { preconditioned} strategies are then
referred to as  $\geo{D} $-MINRES, $\geo{T} $-GMRES and $\geo{C}$-GMRES.

% We postpone the analysis of the strategy proposed in
% this section to a further article. In  Section \ref{sec:num_res} of the present  paper  we
% will show  that it significantly improves the performance of the preconditioners
% especially in case of non-trivial geometries.

\B

\section{Numerical results}\label{sec:num_res}
We  present here numerical experiments to show the performance of our preconditioning strategies.
All the tests are performed by  Matlab (version 8.5.0.197613 R2015a) and using the GeoPDEs toolbox \cite{Vazquez2016},  on a  Intel Xeon i7-5820K processor, running at 3.30 GHz, and with 64 GB of RAM. %The machine has 12 cores in total.
We restrict our tests to a single computational thread. 
%However  we force a single-core sequential execution in all our tests. Indeed, even though our strategy is easily parallelizable on a multicore hardware (which  is done automatically in a MATLAB code as outlined below), 
Indeed, even though our strategy would likely benefit from parallelization on a multicore hardware, as its main computational efforts are matrix products, a careful analysis of the parallel implementation would
 require an in-depth study, which is beyond the scope of this work. 

In the construction and application  of our preconditioner  the two dominant
steps are the eigendecomposition of the univariate matrices (step 1 in
Algorithm~\ref{al:3Ddirect_P}) and the multiplication of Kronecker
matrices (steps 2 and 4  in Algorithm~\ref{al:3Ddirect_P}). These two
key operations are  performed by the \texttt{eig} Matlab function and
by the Tensorlab toolbox \cite{Sorber2014}, respectively. The partial
inclusion of the geometry has a negligible cost (see the Appendix). The tolerance of both MINRES and GMRES is set to $10^{-8}$
and the initial guess is the null vector in all tests. 

As a comparison, we consider a block-diagonal preconditioner
based on an incomplete Cholesky factorization.
In our case,  the zero-fill  incomplete Cholesky factorization,
denoted IC(0),  is computed  by the 
MATLAB \texttt{ichol} routine for the matrix 
$$
\begin{bmatrix}
A_{11} &  0 & 0& 0 \\
0 &  A_{22} & 0 & 0\\
0&  0 & A_{33}&   0\\
0 & 0 & 0 & Q
\end{bmatrix}
$$
and then used in a  Conjugate Gradient (CG) inner iteration   in order to
approximate  the application of the  ideal preconditioner 
\begin{equation}
  \label{eq:ideal-prec}
  \begin{bmatrix}
A_{11} &  A_{12} & A_{13}& 0\\
A_{21} &  A_{22} & A_{23}& 0\\
A_{31} &  A_{32} & A_{33}& 0\\
0 & 0 & 0 & Q
\end{bmatrix}.
\end{equation}
This strategy is  denoted IC(0)-MINRES. 
The  tolerance of this inner CG  loop is set to $10^{-2}$ as this maximizes
the efficiency of the overall strategy %at least 
in the numerical tests we
consider below. The inner  loop is needed to achieve
robustness with respect to $h$,   while robustness with respect to $p$
is common for incomplete factorizations. For this reason, incomplete
factorizations are  often adopted in IGA as  preconditioners: in the
context of the Stokes system, see \cite{Cortes2015} where a similar
approach is considered and benchmarked.

We remark that the geometry parametrization, without simplifications, is directly incorporated in the preconditioner \eqref{eq:ideal-prec}.
 Therefore, as it is seen  in the
tests below, IC(0)-MINRES behaves quite robustly with 
respect to the geometry  parametrizations (since 
%$\kappa\left( Q^{-1} BA^{-1}B^T\right)$ 
$ \lambda_{\max}\left(Q^{-1} BA^{-1}B^T \right)$ and $\lambda_{\min}\left(Q^{-1} BA^{-1}B^T\right) $
depend on $\Omega$, some
dependence on the shape of the domain is unavoidable), while  the
geometry parametrization has a  critical role in our
strategies. Also for this reason,  IC(0)-MINRES is an important term of
comparison.

We consider three different geometries, with increasing complexity
(from the point of view of the geometry parametrization): the
cube, the { eighth} of annulus, and a hollow torus with an eccentric annular cross-section \B (see 
Figure \ref{fig:Geometries}).

{
As discussed in Section \ref{sec:Stokes}, the Stokes problem is discretized using the spaces $V_{h,0}^{TH}$, $Q_{h,0}^{TH}$, $V_{h,0}^{RT}$ and $Q_{h,0}^{RT}$ defined respectively in \eqref{eq:TH_V_basis}, \eqref{eq:TH_Q0}, \eqref{eq:RT_V_basis} and \eqref{eq:RT_Q0}.
In all our tests we choose a uniform regularity $\vect{\alpha} = (\alpha,\alpha,\alpha)$ with $\alpha = p-1$,
except for the hollow torus domain where the spaces are  $C^0$ at
the boundary of the initial mesh elements, and $C^\alpha$,   $\alpha =
p-1$, once the mesh is refined. Note that $p$ always refers to the spline degree of the pressure space.
For Raviart-Thomas discretizations we choose $C_{pen} = 5(\alpha + 1)$ in \eqref{eq:theta}, as it numerically leads to stable schemes (see \cite{Evans2013}). }

Tables \ref{tab:cube_TH}--\ref{tab:Torus_GMRES}  report 
the total solving time, which includes the 
preconditioner setup and the MINRES/ GMRES iterations.  
However, we exclude the time for the formation
of the pressure mass matrix $Q$, which is needed in IC(0) 
and  $\geo{D}$, $\geo{T}$, $\geo{C}$  setup (though only the main
diagonal of  $Q$ is needed in our approaches,  and, in all cases,  only a low-order
approximation of $Q$ is needed for preconditioning). Indeed, it is well known that
the formation of isogeometric matrices is expensive unless ad-hoc routines are adopted (e.g. the weighted-quadrature approach
\cite{Calabro2017} or the low-rank approach \cite{Mantzaflaris2017}). In this
paper, we only  focus on the solver and do not address the efficient
formation of the matrix.  We denote by  $n_{el}$ the number of
elements in each parametric direction.   The symbol  ``$\ast$'' denotes the
impossibility of formation of the matrix $\mathcal{A}$,  due to memory requirements.

In Table \ref{tab:thick_annulus_TIME} we report,
only for the  eighth of annulus  testcase, the preconditioner
setup time and the  preconditioner application time, separately, and  in
Table  \ref{tab:appl_matrix} we report the  percentage   of computing time spent in the preconditioner application. 
  Finally, Table \ref{tab:nu_var} contains number of iterations and solving times obtained with three different choices of variable kinematic viscosity $\nu$ in the hollow torus domain. 
\begin{figure}
 \centering
 \subfloat[][Cube.\label{fig:cube}]
   {\includegraphics[width=.44\textwidth]{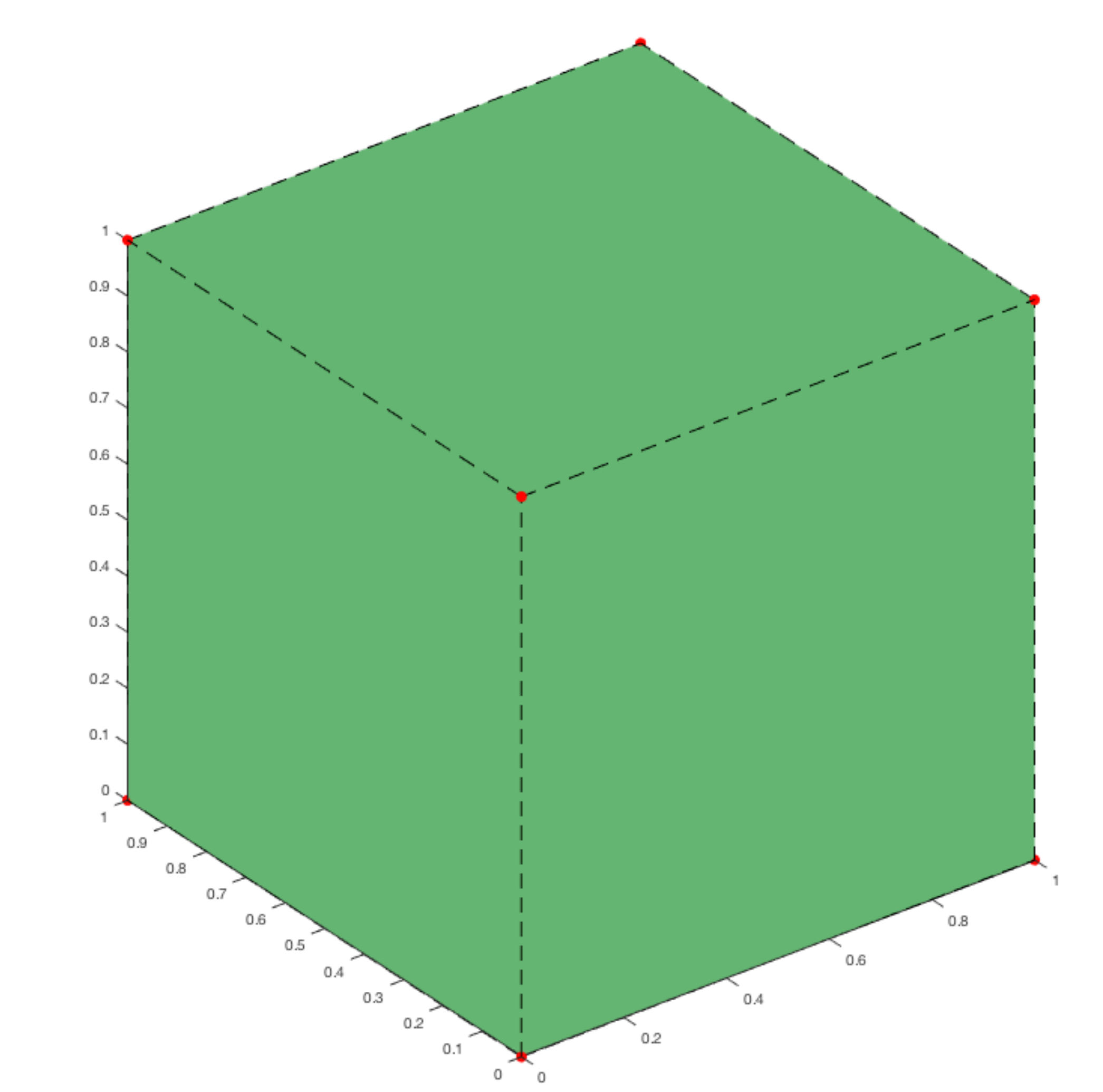}}\quad
 \subfloat[][Eighth of thick annulus.\label{fig:thick_ann}]
   {\includegraphics[width=.44\textwidth]{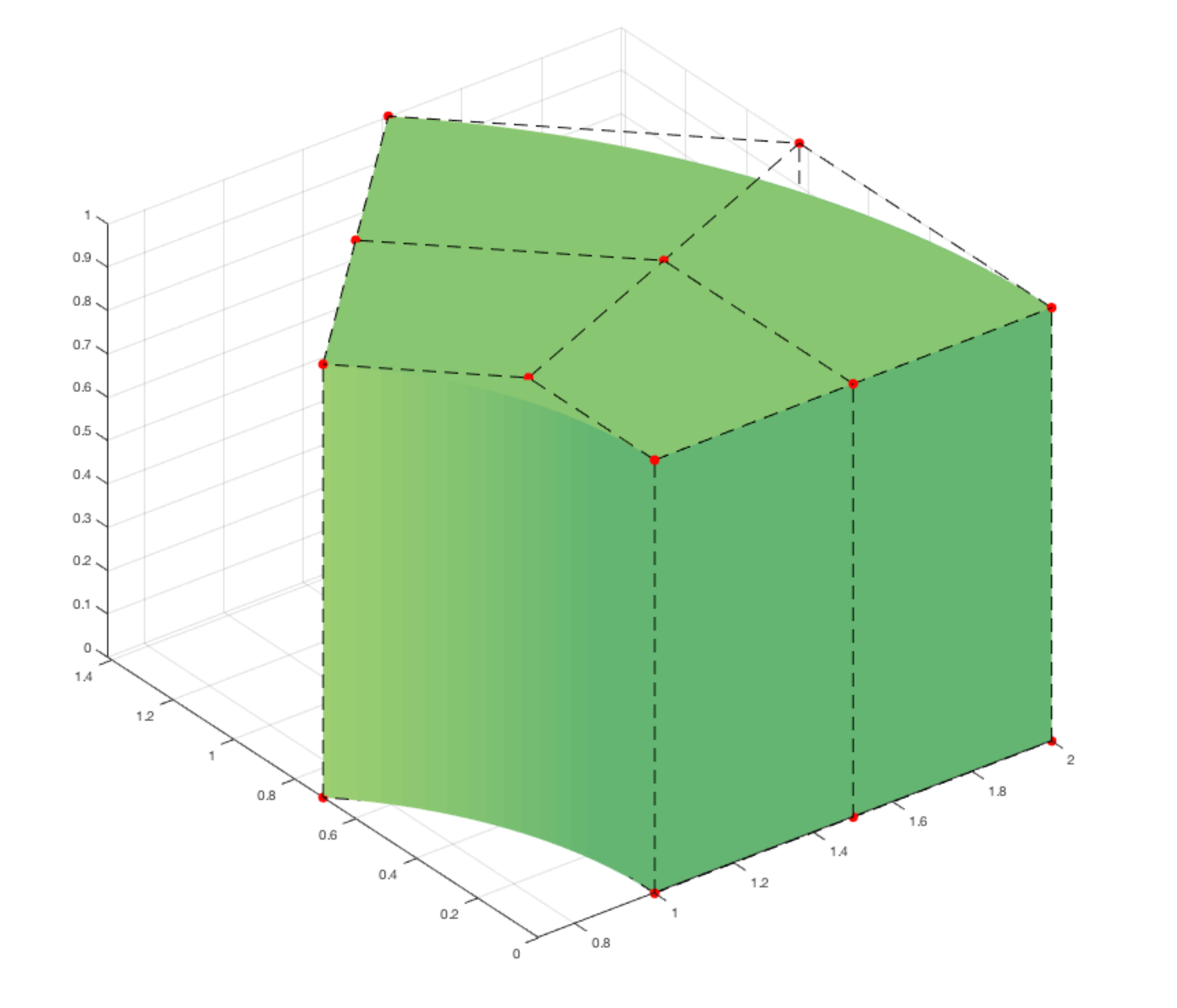}}\quad
 \subfloat[][Hollow torus.\label{fig:Torus}]
   {\includegraphics[width=.44\textwidth]{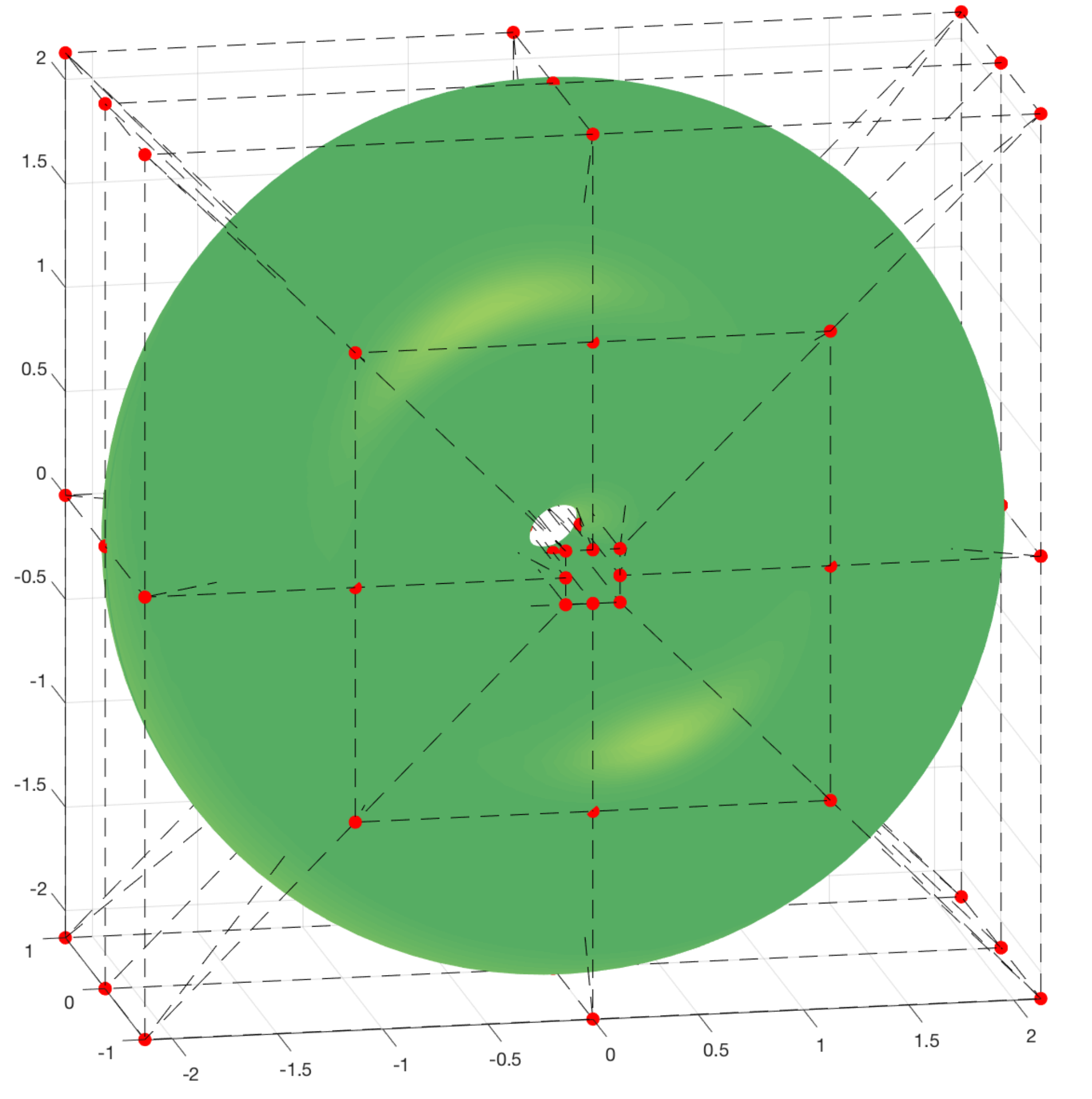}}\quad
 \subfloat[][Hollow torus (cross section).\label{fig:Torus_section}]
   {\includegraphics[width=.44\textwidth]{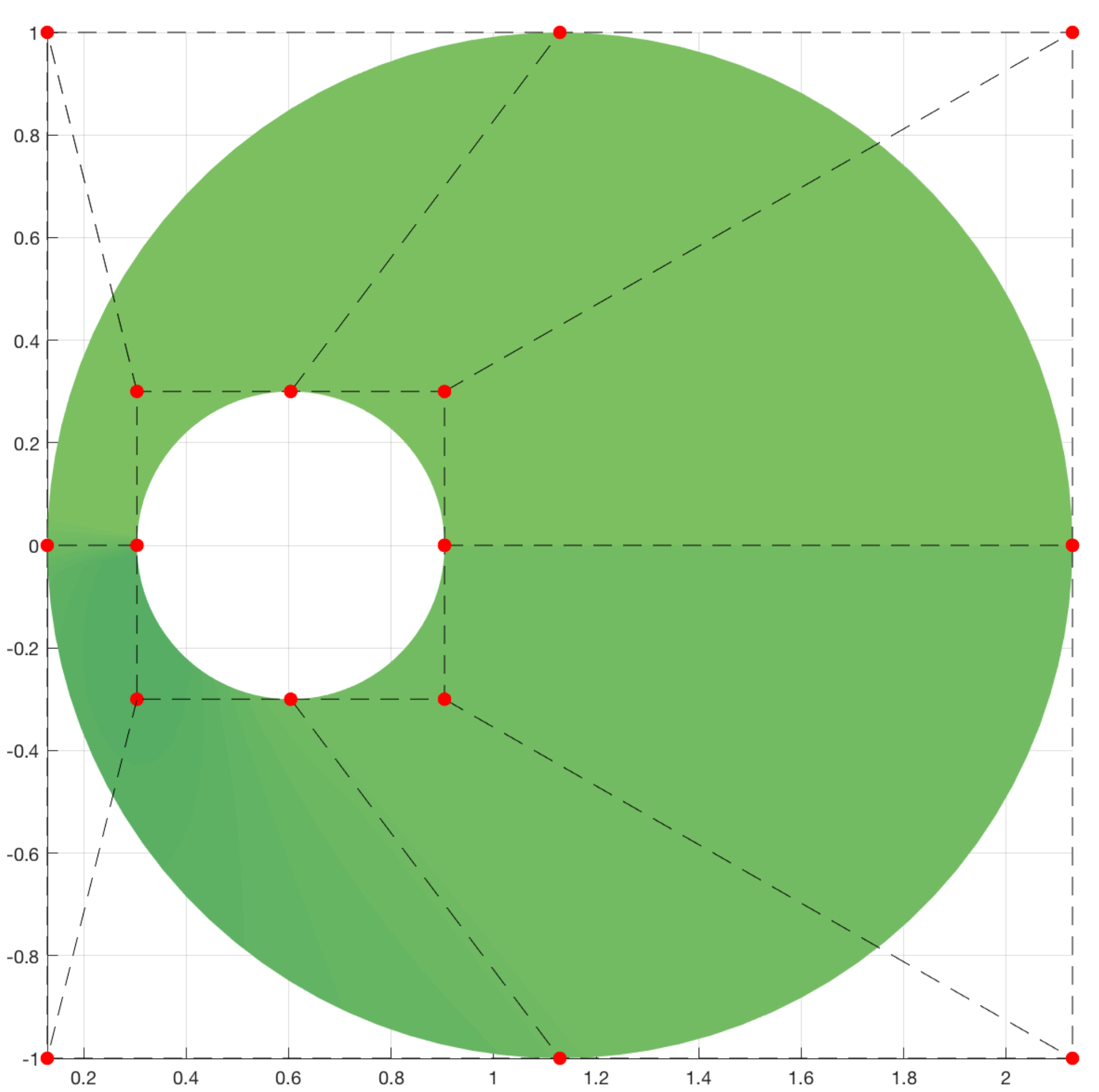}}
 \caption{Computational domains.  }
 \label{fig:Geometries}
\end{figure}

\paragraph{Cube}
We first consider the symmetric driven cavity problem in $\Omega =\widehat{\Omega}=[0,1]^3$ (Figure \ref{fig:cube}). 
In this case, $\F$ is the identity  map and therefore $A_{kk}=\precVc{k}$. 
Homogeneous boundary conditions for the velocity on the lateral sides of the cube and a velocity equal to $[1, 0, 0]^T$ at the top and  to $[-1, 0, 0]^T$ at the bottom are imposed, while $f$ is the null function and $\nu=1$. 

In Table \ref{tab:cube_TH} we report, for the TH discretization,
$\mathcal{P}_{D}$-MINRES and  IC(0)-MINRES performances.
The former is much faster, especially for high degree.
$\mathcal{P}_{D}$-MINRES  results with RT discretization are reported in Table \ref{tab:cube_RT}.
The computational time is lower compared to TH discretization since, for
equal { mesh sizes},   the TH velocity space is  about $2^3$ times
bigger than the one for RT. 
In all cases the number of iterations is uniformly bounded with
respect to  $p$ and $n_{el}$.
  
{\renewcommand\arraystretch{1.2} 
\begin{table}
\begin{center}
\footnotesize \GS
\begin{tabular}{|r|c|c|c|c|}
\hline 
& \multicolumn{4}{|c|}{(TH)\quad $\mathcal{P}_{D}$-MINRES  \quad Iterations / Time (sec)} \\
\hline
$n_{el}$ & $p=2$ & $p=3$ & $p=4$ & $p=5$  \\
\hline
4 & 48  / \z0.16   & 51 / \z\z0.21   & 52 /  \z0.43   & 52 / \z0.81 \\
\hline
8 & 53 / \z0.74   & 53 / \z\z1.49  & 53 /  \z3.01  & 53 / \z5.70    \\
\hline
16 & 56 / \z5.61   & 56 / \z12.76   & 56 /  26.54   & 56 / 51.00   \\
\hline
32 & 56 / 52.23    & 56 / 114.07  & $\ast$   & $\ast$   \\
\hline
\end{tabular}  

\vspace{0.5cm}

\begin{tabular}{|r|c|c|c|c|}
\hline
& \multicolumn{4}{|c|}{(TH)\quad  IC(0)-MINRES  \quad Iterations / Time (sec)} \\
\hline
$n_{el}$ & $p=2$ & $p=3$ & $p=4$ & $p=5$  \\
\hline
4 & 35  / \z\z0.22   & 37 / \z\z0.69  & 37 /  \z\z1.71   & 37 / \z\z  3.77 \\
\hline
8 & 34 / \z\z2.82 & 37 / \z\z7.22   & 35 /  \z16.10   & 36 / \z33.76    \\
\hline
16 & 35 / \z35.09   & 35 / \z74.34   & 35 /  151.87   & 35 / 305.90   \\
\hline
32 & 36 / 482.25   & 36 / 902.51   & $\ast$  & $\ast$   \\
\hline
\end{tabular} 
\caption{Cube domain (TH). Performance of $\mathcal{P}_{D}$-MINRES (upper table) and  IC(0)-MINRES (lower table). }
\label{tab:cube_TH}
\end{center}\B
\end{table}}

{\renewcommand\arraystretch{1.2} 
\begin{table} \GS
\begin{center}
\footnotesize
\begin{tabular}{|r|c|c|c|c|}
\hline
& \multicolumn{4}{|c|}{(RT) \quad $\mathcal{P}_{D}$-MINRES \quad Iterations / Time (sec)} \\
\hline
$n_{el}$ & $p=2$ & $p=3$ & $p=4$ & $p=5$  \\
\hline
4 & 43  / 0.13   & 46 / \z0.18  & 48 / \z0.23  & 48 / 0.39    \\
\hline
8 & 54 / 0.23   & 52 / \z0.44  & 52 / \z0.85 & 52 / 1.59  \\
\hline
16 & 55 / 0.95    & 53 / \z2.56   & 52 / \z4.77   & 52 / 9.02     \\
\hline
32 & 55 / 6.39  & 54 / 16.67   & 52 / 34.58  & $\ast$   \\
\hline
\end{tabular}  
\caption{Cube domain (RT). Performance of $\mathcal{P}_{D}$-MINRES. }
\label{tab:cube_RT}
\end{center}\B
\end{table}}

\paragraph{{ Eighth} of  thick annulus}
Now we consider the { eighth} of a thick  annulus domain (Figure \ref{fig:thick_ann}). 
The internal radius and the height are equal to 1, while  the external radius is equal to 2. 
The boundary data represent a generalization of the symmetric driven cavity boundary conditions, i.e. the velocity is constrained to be  $[-1,   0,   0]^T$ on the set  $\lbrace y = 0 \rbrace$ and $[\sqrt{2} /  2, \sqrt{2}/  2,   0]^T$  on the opposite side, while  homogeneous boundary conditions are imposed anywhere else.
Note that in this case $A_{kk}\neq \precVc{k}$.
 The kinematic viscosity $\nu$ is  constant and equal to 1.

Table  \ref{tab:thick_TH} shows the results of $\mathcal{P}_{D}$-MINRES, $\geo{D}$-MINRES and IC(0)-MINRES  for TH discretization.
Again,   IC(0)-MINRES  is not competitive with
$\mathcal{P}_{D}$-MINRES and $\geo{D}$-MINRES in terms of computing time.
The use of $\geo{D}$-MINRES halves the number of iterations and the solving time  w.r.t.  $\mathcal{P}_{D}$-MINRES, indicating that the inclusion of some geometry information improves the performance of the preconditioner. 
In Table \eqref{tab:thick_annulus_RT} we report results for $\mathcal{P}_{D}^{\F}$-MINRES with RT discretization.
The performances of $\geo{T}$-GMRES and $\geo{C}$-GMRES with TH and RT discretizations are reported in Table \ref{tab:thick_GMRES} and Table \ref{tab:thick_GMRES_RT} respectively.
{\GS
We do not report results for $\mathcal{P}_{T}$-GMRES and $\mathcal{P}_{C}$-GMRES, as the effect of not including any geometry in the preconditioners is similar to the case of the block diagonal preconditioner.
We see that, though the number of iterations of both $\geo{T}$-GMRES and $\geo{C}$-GMRES is lower than $\geo{D}$-MINRES, they are comparable to it in terms of CPU time. This is due the higher application cost of the block triangular and constraint preconditioners 
(which is mainly related to the matrix-vector products with $B$ and $B^T$).
We emphasize that, again, in all the FD-based strategies the number of iterations is uniformly bounded with respect to  $p$ and $n_{el}$.
}

In order to better understand the behaviour of the preconditioners,
and identify directions of further improvements, we analyse in  Table
\ref{tab:thick_annulus_TIME} the computational costs for the setup
and the application of the preconditioners. We recall that for
IC(0)-MINRES, the application corresponds to the execution of the
inner CG iterative solver with residual tolerance  $10^{-2}$. In all
cases, we assume the pressure mass matrix $Q$ is given.  
{\GS 
Table \ref{tab:thick_annulus_TIME} reports the total time spent in the
preconditioner setup and application. We clearly see that the FD-based
preconditioners are much faster than the incomplete
factorization. Note that the setup time for $\mathcal{P}_{D}^{\F}$ is
higher than for $\mathcal{P}_{D}$ due to the cost of computing the
separable approximation of the geometry (see the Appendix): further  studies and tune up of this procedure will be considered in
our following works.

In Table \ref{tab:appl_matrix}, preconditioner application time is compared with the overall computation time of the iterative solver. With $\mathcal{P}_{D}^{\F}$-MINRES strategy, the percentage of time spent for the preconditioner is negligible,  e.g. when $p=5$ and $n_{el}=16$  it is less than  $1\%$. The computation time is indeed mainly spent in the matrix-vector multiplication. This situation suggests that further improvements could be obtained shifting towards a matrix-free  implementation \cite{Sangalli2017}.

The results of Table \ref{tab:thick_annulus_TIME} and
\ref{tab:appl_matrix} clearly show that the suboptimal asymptotic cost
$O(n_{dof}^{4/3})$ of the preconditioner is not seen in practice, up
to the largest problem tested. Note in particular from 
 Table \ref{tab:thick_annulus_TIME}  that the application times of the FD-based
preconditioners scale with respect to $h$ much better than the
asymptotic cost would suggest. This is due to the high efficiency of
the routines that computes the dense matrix-matrix products that are
the core of the FD method.}

{\renewcommand\arraystretch{1.2} 
\begin{table}
\begin{center}
\footnotesize

\GS

\begin{tabular}{|r|c|c|c|c|}
\hline
& \multicolumn{4}{|c|}{(TH) \quad $\mathcal{P}_{D}$-MINRES  \quad Iterations / Time (sec)} \\
\hline
$n_{el}$ & $p=2$ & $p=3$ & $p=4$ & $p=5$  \\
\hline
4 &  116 / \z\z0.39    &   128 / \z\z0.56  &  137  /  \z1.12    &   146 / \z\z2.14  \\
\hline
8 &  146  / \z\z1.66   &  153  / \z\z4.02    &  158 /  \z8.79    &    160 / \z16.83       \\
\hline
16 &  163 / \z16.53    &  164 / \z38.54    &  165  /  75.95    &   162 /   138.17 \\
\hline
32 &   169  /  181.68     &  166 / 337.37    &  $\ast$     & $\ast$   \\
\hline
\end{tabular}  

\vspace{0.5cm}

\begin{tabular}{|r|c|c|c|c|}
\hline
& \multicolumn{4}{|c|}{(TH) \quad $\geo{D}$-MINRES  \quad Iterations / Time (sec)} \\
\hline
$n_{el}$ & $p=2$ & $p=3$ & $p=4$ & $p=5$  \\
\hline
4 & 65  / \z0.21  & 68 / \z\z0.33   & 69 /  \z0.57   & 72 / \z1.09 \\
\hline
8 &  72  /  \z0.91   & 74 / \z\z2.06  & 74 /  \z4.24   & 75 / \z8.01      \\
\hline
16 & 77 / \z8.11   & 77 / \z18.82   & 77 /  36.70    & 77 / 67.74  \\
\hline
32 & 79 / 90.56  & 79 / 168.60    & $\ast$    & $\ast$   \\
\hline
\end{tabular}  

\vspace{0.5cm}

\begin{tabular}{|r|c|c|c|c|}
\hline
& \multicolumn{4}{|c|}{(TH) \quad IC(0)-MINRES \quad Iterations / Time (sec)} \\
\hline
$n_{el}$ & $p=2$ & $p=3$ & $p=4$ & $p=5$  \\
\hline
4 & 39  / \z\z0.28  & 39 / \z\z\z0.79  & 41  /  \z\z1.64  & 41 / \z\z32.69   \\
\hline
8 & 39 / \z\z3.13   & 39 / \z\z\z7.44   & 39 /  \z16.47   & 39 / \z32.69    \\
\hline
16 & 40 / \z39.44   & 39 / \z\z80.53   & 37 /  157.37   & 37 / 281.24   \\
\hline
32 & 38 / 611.55   & 38 / 1085.21    & $\ast$    & $\ast$   \\
\hline
\end{tabular} 
\caption{{ Eighth} of thick annulus domain  (TH). Performance of $\mathcal{P}_{D}$-MINRES (upper table), $\geo{D}$-MINRES (middle table) and  IC(0)-MINRES  (lower table). }
\label{tab:thick_TH}
\end{center}
\end{table}

\B}

{\renewcommand\arraystretch{1.2} 
\begin{table}
\begin{center}
\footnotesize
\GS

\begin{tabular}{|r|c|c|c|c|}
\hline
& \multicolumn{4}{|c|}{(RT)\quad $\geo{D}$-MINRES \quad Iterations / Time (sec)} \\
\hline
$n_{el}$ & $p=2$ & $p=3$ & $p=4$ & $p=5$  \\
\hline
4 & 59  /  0.22   & 58  / \z0.17    & 62  / \z0.30    &   63 / \z0.54   \\
\hline
8 & 63   /  0.29  &  63  / \z0.58    &  61 / \z1.09   &  64 / \z2.10   \\
\hline
16 &  67  /  1.36    & 65  / \z3.23    & 65  / \z6.37    &  66 /  12.07   \\
\hline
32 &  65 / 8.71   &   66 / 23.73    &  66 /  48.38    & $\ast$   \\
\hline
\end{tabular}  
\caption{{ Eighth} of thick annulus domain (RT). Performance of $\geo{D}$-MINRES. }
\label{tab:thick_annulus_RT}
\end{center}
\end{table}\B}

{\renewcommand\arraystretch{1.2} 
\begin{table}\GS
\begin{center}
\footnotesize
\begin{tabular}{|r|c|c|c|c|}
\hline
& \multicolumn{4}{|c|}{(TH)\quad $\geo{T}$-GMRES  \quad Iterations / Time (sec)} \\
\hline
$n_{el}$ & $p=2$ & $p=3$ & $p=4$ & $p=5$  \\
\hline
4 &  38  / \z0.20   & 42 / \z\z0.28  & 42  /  \z0.56   &  47 / \z1.17    \\
\hline
8 &  41  / \z0.78   & 42 / \z\z1.78   & 43 /  \z4.50   & 45 / \z8.50     \\
\hline
16 & 43 / \z7.57    & 44 / \z17.52  & 45 /  35.43  &  46 / 66.21  \\
\hline
32 & 45 / 76.69    &  46 / 165.72    & $\ast$     & $\ast$   \\
\hline
\end{tabular}  

\vspace{0.5cm}

\begin{tabular}{|r|c|c|c|c|}
\hline
& \multicolumn{4}{|c|}{(TH)\quad $\geo{C}$-GMRES  \quad Iterations / Time (sec)} \\
\hline
$n_{el}$ & $p=2$ & $p=3$ & $p=4$ & $p=5$  \\
\hline
4 &    35 / \z0.21   &     37 / \z\z0.30  &    39 /   \z0.59 &   41 / \z1.15  \\
\hline
8 &   37 / \z0.80  &   38 / \z\z1.77    &   39 /  \z4.33    &  41 / \z8.25    \\
\hline
16 &  38 / \z7.19    &   39 / \z16.51   &   40 /  33.47   &  41 /     62.98 \\
\hline
32 &  39 / 61.29 &  40 / 152.44     & $\ast$    & $\ast$   \\
\hline
\end{tabular} 
\caption{{ Eighth} of thick annulus domain (TH). Performance of  $\geo{T}$-GMRES (upper table) and   $\geo{C}$-GMRES (lower table). }
\label{tab:thick_GMRES}
\end{center}\B
\end{table}}

{\renewcommand\arraystretch{1.2} 
\begin{table}
\GS
\begin{center}
\footnotesize
\begin{tabular}{|r|c|c|c|c|}
\hline 
& \multicolumn{4}{|c|}{(RT)\quad $\geo{T}$-GMRES \quad Iterations / Time (sec)} \\
\hline
$n_{el}$ & $p=2$ & $p=3$ & $p=4$ & $p=5$  \\
\hline
4 &   41  / \z0.19   &   44  / \z0.20    &   46  / \z\z0.35   &  48  /  \z0.69  \\
\hline
8 &    46  / \z0.34   &   47   / \z0.71    &  49 / \z\z1.48    &  50  / \z5.55   \\
\hline
16 &  47  / \z1.72  &  49 / \z7.77   &   50  / \z16.57   &   52   /   32.86   \\
\hline
32 &  48 / 21.15    &    50  / 56.50    &   52  / 120.06    & $\ast$   \\
\hline
\end{tabular}  

\vspace{0.5cm}     

\begin{tabular}{|r|c|c|c|c|}
\hline
& \multicolumn{4}{|c|}{(RT)\quad $\geo{C}$-GMRES \quad Iterations / Time (sec)} \\
\hline
$n_{el}$ & $p=2$ & $p=3$ & $p=4$ & $p=5$  \\
\hline
4 &    37  / \z0.19    &     38   / \z0.22  &   39   / \z\z0.36    &   40   / \z0.68  \\
\hline
8 &  38  / \z0.34   &   40  / \z0.71     &    41   /  \z\z1.41  &      42  / \z4.98  \\
\hline
16 &    39  / \z1.63    &  40  / \z6.81    &   41   / \z14.42      &   42   / 28.18  \\
\hline
32 &   39 / 18.30   &    40   / 48.05   &   41 / 100.99    & $\ast$   \\
\hline
\end{tabular} 
\caption{{ Eighth} of thick annulus domain  (RT). Performance of $\geo{T}$-GMRES (upper table) and   $\geo{C}$-GMRES (lower table). }
\label{tab:thick_GMRES_RT}
\end{center} 
\end{table}}

{\renewcommand\arraystretch{1.2} 

\begin{table}
\begin{center}\GS
\footnotesize
\begin{tabular}{|r|c|c|c|c|}
\hline
& \multicolumn{4}{|c|}{$\mathcal{P}_{D}$ Setup times / Total Application times ($\mathcal{P}_{D}$-MINRES)} \\
\hline
$n_{el}$ & $p=2$ & $p=3$ & $p=4$ & $p=5$   \\
\hline
4 &   0.02 / 0.19   & 0.02  / 0.20   &     0.02  / 0.20    &  0.03 /       0.21 \\
\hline
8 &  0.04  / 0.27   &  0.04 / 0.29   & 0.04  /  0.33      &  0.04   /   0.37 \\
\hline
16 &   0.05 / 0.87  &  0.06  / 0.95   &  0.06 /  1.09     &  0.06  /   1.18 \\
\hline
32 &   0.09 / 7.21      & 0.12  / 9.94   &  $\ast$  & $\ast$  
 \\
\hline
\end{tabular} 
\vspace{0.5cm}

\begin{tabular}{|r|c|c|c|c|}
\hline
& \multicolumn{4}{|c|}{$\mathcal{P}_{D}^{\F}$  Setup times /  Total Application times ($\mathcal{P}_{D}^{\F}$-MINRES)} \\
\hline
$n_{el}$ & $p=2$ & $p=3$ & $p=4$ & $p=5$  \\
\hline
4 &   0.05  / 0.88   &  0.06  /  0.10  &  0.06  /  0.10    &  0.07 / 0.11   \\
\hline
8 & 0.09  / 0.13   &  0.12  / 1.49    &    0.16 /  0.16   &  0.21 / 0.18  \\
\hline
16 &  0.28  / 0.46   &  0.49 / 0.51  & 0.76  /  0.56   & 1.14  / 0.62    \\
\hline
32 &   1.57  /  3.86    & 3.20 /  3.93 &  $\ast$  & $\ast$    \\
\hline
\end{tabular}  
\vspace{0.5cm}

\begin{tabular}{|r|c|c|c|c|}
\hline
& \multicolumn{4}{|c|}{IC(0)  Setup times /  Total Application times (IC(0)-MINRES)} \\
\hline
$n_{el}$ & $p=2$ & $p=3$ & $p=4$ & $p=5$   \\
\hline
4  &  0.01  /  \z\z0.21   &  \z0.03   / \z\z0.59   & \z\z0.12   /  \z\z1.43    &  \z0.38 / \z\z3.04   \\
\hline
8 &  0.09  / \z\z2.55  &   \z0.45  /  \z\z6.02   &  \z\z1.46  /   \z13.05   & \z4.23  /  \z23.98 \\
\hline
16 & 0.94  / \z34.49    &  \z4.36 / \z66.68     &   \z13.90  /   125.35    & 40.91    /  207.12    \\
\hline
32 &  9.09 / 558.27    &  46.65 / 889.03  &  $\ast$   & $\ast$     \\
\hline
\end{tabular} 
\vspace{0.5cm}

\caption{Eight of thick annulus domain (TH). Setup times and total application
  times of the preconditioners  $\mathcal{P}_{D}$ (top table), $\mathcal{P}_{D}^{\F}$ (middle table) and  IC(0)  (bottom table). }
\label{tab:thick_annulus_TIME}
\end{center}
\end{table}}

{\renewcommand\arraystretch{1.2} 

\begin{table}
\begin{center}\GS
\footnotesize
\begin{tabular}{|r|c|c|c|c|c|}
\hline
& \multicolumn{4}{|c|}{$\mathcal{P}_{D}^{\F}$} \\
\hline
$n_{el}$ & $p=2$ & $p=3$ & $p=4$ & $p=5$   \\
\hline

8 &  14.28\%  &   \z6.79\%  &  \z3.77\%  &  2.24\%   \\
\hline
16 &   \z5.67\%  &   \z2.70\%   &  \z1.52\%  &  0.91\%     \\
\hline
32 &  \z4.26 \%  &   \z2.33\% &  $\ast$ &  $\ast$    \\
\hline
\end{tabular}  
\vspace{0.5cm}

\begin{tabular}{|r|c|c|c|c|c|}
\hline
& \multicolumn{4}{|c|}{IC(0)} \\
\hline
$n_{el}$ & $p=2$ & $p=3$ & $p=4$ & $p=5$  \\
\hline 
8 &    81.46\% &   80.91\%  &    79.23\%   &  73.35\%  \\
\hline
16 &    87.44\%  &   82.80\%  &    79.65\%    &   73.64\%     \\
\hline
32 &  91.28\%  &  81.92\% &  $\ast$  & $\ast$     \\
\hline
\end{tabular} 
\vspace{0.5cm}

\caption{Eight of thick annulus domain (TH). Percentage of computing time of the preconditioner application in each
MINRES iteration: 
$\mathcal{P}_{D}^{\F}$ (top table) and  IC(0)  (bottom table).}
\label{tab:appl_matrix}
\end{center}
\end{table}}

\GS
\paragraph{Hollow torus} 
The last domain examined is a torus with a hole (Figure \ref{fig:Torus}), obtained by revolving an eccentric annulus (Figure \ref{fig:Torus_section}) around the $y$ axis.
We take $f=\left[ \cos(\arctan(x/z)),\ \sin(4\pi x),\ \sin(\arctan(x/z)) \right]^T$, $\nu=1$ and we impose homogeneous Dirichlet boundary conditions anywhere on the external boundary. We consider here the periodic setting, imposing $C^0$ periodic continuity in the function space.
For this problem, we present only TH discretization results and  focus on the effects of the geometry parametrization on the performances of the preconditioning strategies.
Computing time and number of iterations of $\mathcal{P}_{D}$-MINRES, $\geo{D} $-MINRES and IC(0)-MINRES are reported in Table \ref{tab:Torus_TH}. 
As expected, the geometry  parametrization of the hollow torus has a non-negligible influence on   the performance of our preconditioners.

This is especially true for the $\mathcal{P}_{D}$-MINRES strategy, that requires thousands of iterations to converge.
On the other hand, this influence is greatly reduced with partial inclusion of the geometry
($\mathcal{P}_{D}^{\F}$-MINRES). Here the number of iterations and the CPU times are two orders of magnitude lower 
than for $\mathcal{P}_{D}$-MINRES. CPU times for $\mathcal{P}_{D}^{\F}$-MINRES are also significantly better than for IC(0)-MINRES, despite the fact the number of iterations is higher.
Finally, we remark that the number of iterations for $\mathcal{P}_{D}^{\F}$-MINRES is only three times higher than $\mathcal{P}_{D}$-MINRES on the cube.

\B

{\renewcommand\arraystretch{1.2} 
\begin{table}\GS
\begin{center}
\footnotesize
\begin{tabular}{|r|c|c|c|c| }
\hline
& \multicolumn{4}{|c|}{(TH)\quad  $\mathcal{P}_{D}$-MINRES \quad Iterations / Time (sec)} \\
\hline
$n_{el}$ & $p=2$ & $p=3$ & $p=4$ & $p=5$  \\
\hline
4 &   \z2004 / \z\z\z\z6.42    &   \z4125 / \z\z\z39.16     & \z6411  /     \z153.95   &   \z8305  /   \z\z478.69     \\
\hline
8 &    \z5524   / \z\z\z80.73      &  \z7875 / \z\z360.15  &  \z9914 /      1117.12  &   11032  / \z3286.67    \\
\hline
16 &  \z9931 / \z1081.01     &  11780 /  \z3763.90     &   12964 /  8776.73   &   13553 /  18626.03      \\
\hline
32 &  12864 /  10244.45  &  13426 / 29344.81   & $\ast$   & $\ast$   \\
\hline
\end{tabular}  

\vspace{0.5cm}

\begin{tabular}{|r|c|c|c|c|}
\hline
& \multicolumn{4}{|c|}{(TH)\quad  $\mathcal{P}_{D}^{\F}$-MINRES \quad Iterations / Time (sec)} \\
\hline
$n_{el}$ & $p=2$ & $p=3$ & $p=4$ & $p=5$   \\
\hline
4 &   \z77 / \z\z0.31   & \z87 / \z\z0.89  &   \z97 / \z\z2.59     &  104 /  \z\z6.24  \\
\hline
8 &    \z96 / \z\z1.52     &  104 / \z\z4.99  &  110 / \z12.82     &  115  /  \z34.70  \\
\hline
16 &  119 / \z13.87   &   124 / \z40.89    &  133 /  \z91.82  &   139 / 197.30   \\
\hline
32 &  142 / 116.95  & 147 / 344.34   &   $\ast$ & $\ast$   \\
\hline
\end{tabular}  

\vspace{0.5cm}

\begin{tabular}{|r|c|c|c|c|}
\hline
& \multicolumn{4}{|c|}{(TH)\quad  IC(0)-MINRES \quad Iterations / Time (sec)} \\
\hline
$n_{el}$ & $p=2$ & $p=3$ & $p=4$ & $p=5$  \\
\hline
4 &     49  / \z\z1.05 &    46 / \z\z\z3.74    &  50 / \z11.79    &  50 /  \z31.42  \\
\hline
8 &    45 / \z\z5.42   &    45 / \z\z18.52    &  45  /  \z51.18       &   45  / 126.83  \\
\hline
16 &   45  / \z45.11  &   43  / \z125.60        &   45 /  307.79   &  45 /  660.63 \\
\hline
32 &   45 /  493.12     &  44  / 1352.81  & $\ast$ & $\ast$   \\
\hline
\end{tabular} 
\caption{Hollow torus domain (TH). Performance of $\mathcal{P}_{D}$-MINRES (upper table), $\mathcal{P}_{D}^{\F}$-MINRES (middle table) and  IC(0)-MINRES  (lower table). }
\label{tab:Torus_TH}
\end{center}
\end{table}}

{\renewcommand\arraystretch{1.2} 
\begin{table}\GS
\begin{center}
\footnotesize
\begin{tabular}{|r|c|c|c|c|}
\hline
& \multicolumn{4}{|c|}{(TH)\quad $\mathcal{P}_{T}^{\F}$-GMRES  \quad Iterations / Time (sec)} \\
\hline
$n_{el}$ & $p=2$ & $p=3$ & $p=4$ & $p=5$  \\
\hline
4 &     44 / \z\z0.30    &  50 / \z\z0.80     &   57 / \z2.39  &   61  /   \z\z6.89  \\
\hline
8 &    49 / \z\z1.25    &   54 /  \z\z4.54    &  58  / 11.98   &    62 /  \z31.52  \\
\hline
16 &  58 /  \z10.78    &  60  / \z32.86     &  63 / 73.52  &    67 /  159.46  \\
\hline
32 &  68 /  105.31 &  71  /  275.54       &  $\ast$   & $\ast$  \\
\hline
\end{tabular}  

\vspace{0.5cm}

\begin{tabular}{|r|c|c|c|c|}
\hline
& \multicolumn{4}{|c|}{(TH)\quad $\mathcal{P}_{C}^{\F}$-GMRES \quad Iterations / Time (sec)} \\
\hline
$n_{el}$ & $p=2$ & $p=3$ & $p=4$ & $p=5$ \\
\hline
4 &     37 / \z\z0.28   &  41 / \z\z0.74     &   45 / \z\z2.09    &   50  / \z\z6.09  \\
\hline
8 &   41  / \z\z1.16    &  45  / \z\z4.07    &  49  /  \z10.82   &    53 / \z28.59     \\
\hline
16 &  51 / \z10.27   &   55 /  \z31.73    & 59  / \z72.91     &   63  /  158.12   \\
\hline
32 & 69  / 113.81  &  72  / 299.62   &  $\ast$   & $\ast$  \\
\hline
\end{tabular}  
\caption{Hollow torus domain (TH). Performance of $\mathcal{P}_{T}^{\F}$-GMRES (upper table) and  $\mathcal{P}_{C}^{\F}$-GMRES  (lower table) }
\label{tab:Torus_GMRES}
\end{center}
\end{table}}

\GS
\paragraph{Hollow torus: variable $\nu$} 
In this paragraph we  investigate the effect of a variable kinematic viscosity $\nu$ on our preconditioning strategies.
We consider the hollow torus domain with $\nu = 1 + (k-1)(1 +
\cos(\arctan(x/z)))/2)$ depending on a parameter $k$,  $p=3$ and $n_{el}=32$ and  we compare in Table \ref{tab:nu_var} the performances of $\mathcal{P}_{D}$-MINRES, $\mathcal{P}_{D}^{\F}$-MINRES and $\mathcal{P}_{T}^{\F}$-GMRES. %The symbol ``$-$'' denotes the fact the the solver does not reach the fixed tolerance because of stagnation of two consecutive iterates. 

 $\mathcal{P}_{D}$-MINRES is the worse strategy both in terms of number of iterations and in computing times for all values of $k$ and in the case $k=10000$ it does not even converge.
The geometry inclusion strategy, on the other hand, succeeds in capturing the effect of the variable $\nu$; the number of iterations  of $\mathcal{P}_{D}^{\F}$-MINRES and  $\mathcal{P}_{T}^{\F}$-GMRES remains stable when $k$ varies.

We remark that $\mathcal{P}_{C}^{\F}$-GMRES has  behaviour similar to  $\mathcal{P}_{T}^{\F}$-GMRES, as it is also highlighted in the previous testcases,  and for this reason we do not consider it in the table.

{\renewcommand\arraystretch{1.2} 
\begin{table}\GS
\begin{center}
\footnotesize
\begin{tabular}{|l|c|c|c|}
\hline
  & $\mathcal{P}_{D}$-MINRES  & $\mathcal{P}_{D}^{\F}$-MINRES &  $\mathcal{P}_{T}^{\F}$-GMRES   \\
\hline
$k=1$ &   13426  / 29344.81     &     147 / 344.34    &    71 /  275.54     \\
\hline
$k=100$ &   17254   / 37667.04   &    180 /  400.46 &   84  /  325.02  \\
\hline
$k=10000$ &    $-$    &  180  / 407.68      &   84 /    326.78   \\
\hline
\end{tabular} 
\caption{Hollow torus domain (TH). Performance of $\mathcal{P}_{D}$-MINRES, $\mathcal{P}_{D}^{\F}$-MINRES and  $\mathcal{P}_{T}^{\F}$-GMRES for $p=3$ and $n_{el} = 32$.  The symbol ``$-$'' denotes the fact the the solver does not converge because of stagnation.  }
\label{tab:nu_var}
\end{center}
\end{table}
\B

\section{Conclusions} \label{sec:conclusions}

In this work we have addressed the problem of finding good
preconditioners for isogeometric discretizations of the Stokes system.
Our approach exploits the  tensor-product structure of the multivariate
B-spline basis. The application of our preconditioners
$\mathcal{P}_{D}$,  $\mathcal{P}_{T}$ and  $\mathcal{P}_{C}$ (and
their  coefficients-including  variants $\mathcal{P}_{D}^{\F}$,  $\mathcal{P}_{T}^{\F}$ and  $\mathcal{P}_{C}^{\F}$ ) requires the
solution of linear systems that have a Kronecker structure, or a
Sylvester-like equation structure. This can be performed by
direct solvers with the highest efficiency. This also guarantees
robustness with respect to both the spline degree  $p$ and mesh
resolution. \GS  Numerical tests show that  $\mathcal{P}_{D}^{\F}$,
$\mathcal{P}_{T}^{\F}$ and  $\mathcal{P}_{C}^{\F}$   allow to maintain
the performance also in case of non-trivial geometries and highly
oscillating coefficients. \B

We have performed a comparative numerical benchmarking with respect
to a more common approach which uses a similar block structure for the
preconditioner but  applies it by an   incomplete Cholesky  factorization and
an inner conjugate gradient.
The solution time is always in favour of  our preconditioners, despite that
they are influenced by the geometry parametrization. Even more
important is that our preconditioners are well suited for a matrix-free
approach, which should lead to solvers that are orders of magnitude faster. This
is the most promising research direction that we will consider in the
near future \cite{Sangalli2017}. 

There are other important extensions to this work that   will be the topic of our future researches. Multipatch geometries are possible by combining our
framework to known domain decomposition techniques, e.g. FETI-DP
\cite{Pavarino2016}. A challenging extension is to
the Oseen system, in particular with a dominant transport
term. Finally, we will work on space-time formulations.

\section*{Appendix}
In this appendix we report more details about the separation of variables strategy that we use to include in $\precV$ some information on the geometry. 
A complete analysis of the geometry inclusion strategy will be addressed in a forthcoming work.

We incorporate in $\precV$ some  information on the parametrization  present in the diagonal blocks $A_{kk}$ by making  approximations of the full matrix $\mathfrak{C}_k$ (see equations \eqref{eq:Q_TH}, \eqref{eq:Q_RT}), whose entries are functions of three variables that we denote with $c^k_{ij}(\boldsymbol\eta)$:
$$
\mathfrak{C}_k(\boldsymbol\eta)=\begin{bmatrix}
c^k_{11}(\boldsymbol\eta) & c^k_{12}(\boldsymbol\eta) & c^k_{13}(\boldsymbol\eta)\\
c^k_{21}(\boldsymbol\eta) & c^k_{22}(\boldsymbol\eta) & c^k_{23}(\boldsymbol\eta)\\
c^k_{31}(\boldsymbol\eta) & c^k_{32}(\boldsymbol\eta) & c^k_{33}(\boldsymbol\eta)
\end{bmatrix}.
$$
We discard the off-diagonal terms and approximate the diagonal entries $c^k_{11}(\boldsymbol\eta)$, $c^k_{22}(\boldsymbol\eta)$ and $c^k_{33}(\boldsymbol\eta)$ as follows (by the algorithm in \cite{Diliberto1951, Wachspress1984, Wachspress2013} )

$$
{\mathfrak{C}}_k(\boldsymbol\eta)\approx\accentset{\frown}{\mathfrak{C}}_k(\boldsymbol\eta):=\!\!
\begin{bmatrix}
 \tau_{1}^{k}(\eta_1)\mu_{2}^{k}(\eta_2)\mu_{3}^{k}(\eta_3) & 0 & 0\\
0 & \mu_{1}^{k}(\eta_1)\tau_{2}^{k}(\eta_2)\mu_{3}^{k}(\eta_3) & 0\\
0 & 0 & \mu_{1}^{k}(\eta_1)\mu_{2}^{k}(\eta_2)\tau_{3}^{k}(\eta_3) 
\end{bmatrix}.
$$
\GS The approximation above is computed directly at the quadrature points, hence no function space has to be selected a-priori.
The cost of this algorithm is proportional to the number of quadrature points, hence in our setting it requires $O(n_{el} p^d)$ FLOPs.
This cost could be easily reduced by computing the approximation on a coarser grid of points, and then extending by interpolation.
However this is not necessary, since such cost is already negligible in the context of the iterative procedures considered in this paper, as can be seen e.g. by comparing Tables \ref{tab:thick_TH} and \ref{tab:thick_annulus_TIME}.
 \B
 
Keeping the block-diagonal structure of $\precV$ (cfr. \eqref{eq:precV}),  we define for the TH discretization, $k=1,2,3$ and $i,j=1,...,n_{V,k}^{TH}$
\begin{IEEEeqnarray*}{l}
 \left[\precVggc{k}^{TH}\right]_{i,j}   :=
\int_{\widehat{\Omega}} \  \left(\nabla \Bmull{\vect{p}+1}{\vect{\alpha},i}\right)^T  \, \accentset{\frown}{\mathfrak{C}}_k^{TH}\, \nabla  \Bmull{\vect{p}+1}{\vect{\alpha},j}\;\d\boldsymbol\eta,
\end{IEEEeqnarray*}
 while for the RT discretization, $k=1,2,3$ and $i,j=1,...,n_{V,k}^{RT}$
\begin{IEEEeqnarray*}{lcl}
\left[ \precVggc{k}^{RT}\right]_{i,j}\!\! & := & \int_{\widehat{\Omega}} \!\! \left( \nabla  \Bmull{\vect{p}+\mathbf{e}_k}{\vect{\alpha}+\mathbf{e}_k,i}\right)^T\! \accentset{\frown}{\mathfrak{C}}_k^{RT}\, \nabla \ \Bmull{\vect{p}+\mathbf{e}_k}{\vect{\alpha}+\mathbf{e}_k,j}\;\d\boldsymbol\eta 
  +   2\int_{\partial\widehat{\Omega}}
 \left[\frac{ C_{pen}}{h}\Bmull{\vect{p}+\mathbf{e}_k}{\vect{\alpha}+\mathbf{e}_k,i}\mathbf{e}_k\cdot \left(\accentset{\frown}{\mathfrak{C}}_k^{RT} \mathbf{e}_k\Bmull{\vect{p}+\mathbf{e}_k}{\vect{\alpha}+\mathbf{e}_k,j} \right) \right. \\
& &  -\left. \left(\left(\nabla^s\left(\mathbf{e}_k\Bmull{\vect{p}+\mathbf{e}_k}{\vect{\alpha}+\mathbf{e}_k,i} \right)\hat{\vel{n}}\right)\right)\cdot\left(\accentset{\frown}{\mathfrak{C}}_k^{RT}\mathbf{e}_k\Bmull{\vect{p}+\mathbf{e}_k}{\vect{\alpha}+\mathbf{e}_k,j}\right)-   \left(\left(\nabla^s\left(\mathbf{e}_k\Bmull{\vect{p}+\mathbf{e}_k}{\vect{\alpha}+\mathbf{e}_k,j} \right)\hat{\vel{n}}\right)\right)\cdot\left(\accentset{\frown}{\mathfrak{C}}_k^{RT}\mathbf{e}_k\Bmull{\vect{p}+\mathbf{e}_k}{\vect{\alpha}+\mathbf{e}_k,i}\right) \right] \d\hat{\Gamma} .
%\hat{\sigma}(\accentset{\frown}{\mathfrak{C}}_k^{RT}\mathbf{e}_k\Bmull{\vect{p}+\mathbf{e}_k}{\vect{\alpha}+\mathbf{e}_k,i},\ \accentset{\frown}{\mathfrak{C}}_k^{RT}\mathbf{e}_k\Bmull{\vect{p}+\mathbf{e}_k}{\vect{\alpha}+\mathbf{e}_k,j}).
\end{IEEEeqnarray*} 

The preconditioners $\precVggc{k}$ maintain the tensor structure  of
\eqref{eq:prec3D_TH} and \eqref{eq:prec3D_RT}: \B
%\begin{small}
\begin{IEEEeqnarray*}{c}
%\IEEEyesnumber\label{eq:prec3D_TH_SPD-G} \IEEEyessubnumber*
\precVggc{1}^{TH}    =  K^{1,TH}_3 \!  \otimes\!  M^{1,TH}_2\!   \otimes\!  M^{1,TH}_1   +  M^{1,TH}_3 \! \! \otimes \!  K^{1,TH}_2\!   \otimes\!  M^{1,TH}_1   +    M^{1,TH}_3 \!   \otimes\!  M^{1,TH}_2 \!   \otimes\!  K^{1,TH}_1 ,  \\ 
\precVggc{2}^{TH}    =  K^{2,TH}_3\!  \otimes\!  M^{2,TH}_2\!   \otimes\!  M^{2,TH}_1   +   M^{2,TH}_3  \! \otimes\!  K^{2,TH}_2\!   \otimes\!  M^{2,TH}_1   +  M^{2,TH}_3\!   \otimes\!  M^{2,TH}_2\!   \otimes\!  K^{2,TH}_1  ,  \\ 
\precVggc{3}^{TH}   =   K^{3,TH}_3\!   \otimes\!  M^{3,TH}_2\!   \otimes\!  M^{3,TH}_1   +  M^{3,TH}_3\!   \otimes\!  K^{3,TH}_2\!  \otimes\!  M^{3,TH}_1   +  M^{3,TH}_3\!   \otimes\!  M^{3,TH}_2\!   \otimes\!  K^{3,TH}_1   , 
\end{IEEEeqnarray*}
\begin{IEEEeqnarray*}{c}
%\IEEEyesnumber\label{eq:prec3D_RT_SPD-G} \IEEEyessubnumber*
\precVggc{1}^{RT}   =  \widetilde{K}^{1,RT}_3 \!  \otimes\!  \widetilde{M}^{1,RT}_2\!   \otimes\!  M^{1,RT}_1  +   \widetilde{M}^{1,RT}_3\!   \otimes\!  \widetilde{K}^{1,RT}_2\!   \otimes \! M^{1,RT}_1   +    \widetilde{M}^{1,RT}_3\!   \otimes\!  \widetilde{M}^{1,RT}_2\!   \otimes\!  K^{1,RT}_1  ,  \\ 
\precVggc{2}^{RT}    =  \widetilde{K}^{2,RT}_3\!   \otimes\!  M^{2,RT}_2\!   \otimes\!  \widetilde{M}^{2,RT}_1  +   \widetilde{M}^{2,RT}_3\!  \otimes\! K^{2,RT}_2\!   \otimes\!  \widetilde{M}^{2,RT}_1  +   \widetilde{M}^{2,RT}_3\!  \otimes\!  M^{2,RT}_2\!  \otimes\! \widetilde{K}^{2,RT}_1 ,  \\ 
\precVggc{3}^{RT}    =   K^{3,RT}_3\!   \otimes\!  \widetilde{M}^{3,RT}_2\!   \otimes\!  \widetilde{M}^{3,RT}_1  +   M^{3,RT}_3\!   \otimes\!  \widetilde{K}^{3,RT}_2\!   \otimes\!  \widetilde{M}^{3,RT}_1   + M^{3,RT}_3\!   \otimes\!  \widetilde{M}^{3,RT}_2\!   \otimes\!  \widetilde{K}^{3,RT}_1  ,
\end{IEEEeqnarray*} 
%\end{small}
where, for $d,k=1,2,3$, the  new pairs $(K^d_k, M^d_k)$ and $(\widetilde{K}^d_k,\widetilde{M}^d_k)$  are
%\begin{small}
\begin{IEEEeqnarray*}{ll}
\left[K^{d,TH}_k\right]_{l,s} \!  =& \!  \int_{[0,1]}\!\!\!\!\! \tau_{k}^{d,TH}(\eta_k) (\bunivv{p+1}{\alpha_k,l})'(\eta_k) (\bunivv{p+1}{\alpha_k, s})'(\eta_k)\, \d\eta_k, \\
\left[M^{d,TH}_k\right]_{l,s} \! =& \! \int_{[0,1]}\!\!\!\!\! \mu_{k}^{d,TH}(\eta_k) \bunivv{p+1}{\alpha_k, l}(\eta_k) \ \bunivv{p+1}{\alpha_k, s}(\eta_k)\, \d\eta_k, 
\end{IEEEeqnarray*}
for $l,s=2,...,m_{\alpha_k}^{p+1}-1$, and
\begin{IEEEeqnarray*}{ll}
\left[K_{k}^{d,RT}\right]_{l,s} \!  = &\!  \int_{[0,1]}\!\!\!\!\! \tau_{k}^{d,RT}(\eta_k)(\bunivv{p+1}{\alpha_k +1, l})'(\eta_k) (\bunivv{p+1}{\alpha_k+1, s})'(\eta_k)\, \d\eta_k,  \nonumber\\
\\
\left[M_{k}^{d,RT}\right]_{l,s} \!  = &\! \int_{[0,1]}\!\!\!\!\! \mu_{k}^{d,RT}(\eta_k)\bunivv{p+1}{\alpha_k+1, l}(\eta_k) \ \bunivv{p+1}{\alpha_k+1, s}(\eta_k)\, \d\eta_k, 
\end{IEEEeqnarray*}
for $l,s=2,...,\dimpu{k}\! - \! 1$, and finally
\begin{IEEEeqnarray*}{ll}
 \left[\widetilde{K}_{k}^{d,RT}\right]_{l,s}  = &  \int_{[0,1]}\!\!\!\!\! \tau_{k}^{d,RT}(\eta_k)(\bunivv{p}{\alpha_k, l})'(\eta_k) (\bunivv{p}{\alpha_k, s})'(\eta_k)\, \d\eta_k  -\bigg[ \tau_{k}^{d,RT}(1)(\bunivv{p}{\alpha_k, l})'(1) \bunivv{p}{\alpha_k, s}(1)  \\
 & \qquad - \tau_{k}^{d,RT}(0)(\bunivv{p}{\alpha_k, l})'(0) \bunivv{p}{\alpha_k, s}(0)     + \tau_{k}^{d,RT}(1)(\bunivv{p}{\alpha_k, s})'(1) \bunivv{p}{\alpha_k, l}(1)  \\ 
 & \qquad - \tau_{k}^{d,RT}(0)(\bunivv{p}{\alpha_k, s})'(0) \bunivv{p}{\alpha_k, l}(0)   -2\frac{C_{pen}}{h}\big(\tau_{k}^{d,RT}(1)\bunivv{p}{\alpha_k, l}(1)\bunivv{p}{\alpha_k, s}(1) \\
 & \qquad + \tau_{k}^{d,RT}(0)\bunivv{p}{\alpha_k, l}(0)\bunivv{p}{\alpha_k, s}(0) \big) \bigg], \\
\left[\widetilde{M}_k^{d,RT}\right]_{l,s} = &  \int_{[0,1]}\!\!\!\!\! \mu_{k}^{d,RT}(\eta_k)\bunivv{p}{\alpha_k, l}(\eta_k) \ \bunivv{p}{\alpha_k, s}(\eta_k)\, \d\eta_k,
\end{IEEEeqnarray*}
for $l,s=1,...,\dimp{k}$.

%{\renewcommand\arraystretch{1.2}
%\begin{table} \MM
%\begin{center}
%\footnotesize
%\begin{tabular}{|r|c|c|c|c|}
%\hline
%& \multicolumn{4}{|c|}{(TH)\quad Separation of variable Times } \\
%\hline
%$n_{el}$ & $p=2$ & $p=3$ & $p=4$ & $p=5$  \\
%\hline
%4 & 0.01  &  0.01   &  0.01    &  0.01  \\
%\hline
%8 &  0.01  &  0.04 & 0.05  & 0.08\\
%\hline
%16 &   0.11 &  0.25  & 0.42   & 0.63  \\
%\hline
%32 &   0.97  &  1.76   & 3.03  & $\ast$  \\
%\hline
%\end{tabular}  
%\caption{Eight of thick annulus domain (TH). Separation of variable algorithm times. }
%\label{tab:sep_var_times}
%\end{center}
%\end{table}}

\section*{Acknowledgments}

The authors were partially supported by the European Research Council
through the FP7 Ideas Consolidator Grant \emph{HIGEOM} n.616563.
The authors are members of the  Gruppo Nazionale Calcolo
Scientifico-Istituto Nazionale di Alta Matematica (GNCS-INDAM), and
the third  author was partially supported by GNCS-INDAM for this research.
This support are gratefully acknowledged.
%The authors  are members of the INdAM Research group GNCS.

 \bibliographystyle{plain}
 \bibliography{biblio_preconditioners}
\end{document}